\documentclass[twoside,11pt,reqno]{amsart}

\usepackage{amssymb,amscd,mathrsfs,epic,latexsym,tikz,mathrsfs,cite,hyperref,mathtools,pb-diagram}
\usepackage[matrix,arrow]{xy}
\usepackage{amsmath,wasysym}
\usetikzlibrary{decorations.markings}

\makeatletter

\@addtoreset{equation}{section}

\newcommand{\arxiv}[1]{\href{http://arxiv.org/abs/#1}{\tt arXiv:\nolinkurl{#1}}}
\newcommand{\arXiv}[1]{\href{http://arxiv.org/abs/#1}{\tt arXiv:\nolinkurl{#1}}}

\newtheorem{Proposition}{Proposition}[section]
\newtheorem{Lemma}[Proposition]{Lemma}
\newtheorem{Theorem}[Proposition]{Theorem}
\newtheorem{Corollary}[Proposition]{Corollary}

\newtheorem{Example}[Proposition]{Example}

\newtheorem{Hypothesis}[Proposition]{Hypothesis}

\newbox\squ  
\setbox\squ=\hbox{\vrule width.6pt
   \vbox{\hrule height.6pt width.4em\kern1ex
         \hrule height.6pt}%
   \vrule width.6pt}

\def\FF{\mathbb{F}}

\def\Mod#1{#1\!\operatorname{-Mod}}
\def\mod#1{#1\!\operatorname{-mod}}
\def\Proj#1{#1\!\operatorname{-proj}}

\def\CH{{\operatorname{ch}_q\:}}

\def\Q{{\mathbb Q}}
\def\Z{{\mathbb Z}}

\def\0{{\bar 0}}
\def\1{{\bar 1}}

\def\A{{\mathcal A}}

\def\hom{{\operatorname{hom}}}

\def\End{{\operatorname{End}}}

\def\Hom{{\operatorname{Hom}}}

\newcommand{\tw}{\operatorname{s}}

\newcommand{\DD}{\mathfrak{D}}
\newcommand{\BB}{\mathfrak{B}}
\newcommand{\f}{\mathbf{f}}
\newcommand{\Af}{\prescript{}{\A}{\f}}
\newcommand{\pf}{\prescript{\prime}{}{\f}}
\newcommand{\pAf}{\prescript{\prime}{\A}{\f}}
\newcommand{\ptheta}{\prescript{\prime}{}{\theta}}
\newcommand{\<}{\langle}
\renewcommand{\>}{\rangle}
\newcommand{\Si}{\mathfrak{S}}
\newcommand{\Pol}{P}
\newcommand{\Sym}{\Lambda}

\def\Laurent{\mathcal A}

\def\qdim{{\operatorname{dim}_q}\,}

\def\Ind{{\operatorname{Ind}}}

\def\Res{{\operatorname{Res}}}

\def\im{{\operatorname{im}}}

\def\height{{\operatorname{ht}}}

\def\bi{\text{\boldmath$i$}}
\def\bj{\text{\boldmath$j$}}

\def\eps{{\varepsilon}}
\def\phi{{\varphi}}

\def\ga{{\gamma}}
\def\Ga{{\Gamma}}

\def\de{{\delta}}
\def\De{{\Delta}}
\def\al{{\alpha}}
\def\be{{\beta}}

\def\si{{\sigma}}

\def\iso{\stackrel{\sim}{\longrightarrow}}
\def\underbar{\mathpalette\@underbar}

\def\words{{\langle I \rangle}}
\def\onto{{\twoheadrightarrow}}
\def\into{{\hookrightarrow}}
\def\O{{\mathcal O}}
\def\shift{{\tt sh}}

\colorlet{darkgreen}{green!50!black}
\tikzset{dots/.style={very thick,loosely dotted},
         greendot/.style={fill,circle,color=darkgreen,inner sep=1.5pt,outer sep=0}
}
\def\greendot(#1,#2){\node[greendot] at(#1,#2){}}

\newenvironment{braid}{
  \begin{tikzpicture}[baseline=6mm,blue,line width=1pt, scale=0.4,
                      draw/.append style={rounded corners},
                      every node/.append style={font=\tiny}]%
  }{\end{tikzpicture}
}

\newenvironment{dynkin}{
  \begin{tikzpicture}[baseline=6mm, scale=0.8,
                      draw/.append style={rounded corners},
                      every node/.append style={font=\tiny}]%
  }{\end{tikzpicture}
}

\author[Alexander Kleshchev]{\sc Alexander S. Kleshchev}
\address{Department of Mathematics\\ University of Oregon\\
Eugene\\ OR~97403, USA}
\email{klesh@uoregon.edu}

\author[Joseph Loubert]{\sc Joseph W. Loubert}
\address{Department of Mathematics\\ University of Oregon\\
Eugene\\ OR~97403, USA}
\email{loubert@uoregon.edu}

\thanks{Research supported in part by the NSF grant no. DMS-1161094 and the Humboldt Foundation. Substrantial part of the paper has been completed at the University of Stuttgart. The authors thank Steffen Koenig for hospitality.}

\begin{document}

\begin{abstract}
We prove that the Khovanov-Lauda-Rouquier algebras $R_\al$ of finite type are (graded) affine cellular in the sense of Koenig and Xi. In fact, we establish a stronger property, namely that the affine cell ideals in $R_\al$ are generated by idempotents.  This in particular implies the (known) result that the global dimension of $R_\al$ is finite. 
\end{abstract}

\title[Affine cellularity of finite type KLR algebras]{Affine Cellularity of Khovanov-Lauda-Rouquier Algebras of Finite Types}

\maketitle

\section{Introduction}\label{SIntro}
The goal of this paper is to establish (graded) affine cellularity in the sense of Koenig and Xi \cite{KoXi} for the Khovanov-Lauda-Rouquier algebras $R_\al$ of finite Lie type. In fact, we construct a chain of affine cell ideals in $R_\al$ which are generated by idempotents. This stronger property is analogous to quasi-heredity for finite dimensional algebras, and by a general result of Koenig and Xi \cite[Theorem 4.4]{KoXi}, it also implies finiteness of the global dimension of $R_\al$. Thus we obtain a new proof of (a slightly stronger version of) a recent result of Kato \cite{Kato} and McNamara \cite{McN} (see also \cite{BKM}). As another application, one gets a theory of standard and proper standard modules, cf. \cite{Kato},\cite{BKM}. 
It would be interesting to apply this paper to prove the conjectural (graded) cellularity of cyclotomic KLR algebras of finite types.

Our approach is independent of the homological results in  \cite{McN}, \cite{Kato} and \cite{BKM} (which relies on \cite{McN}). The connection between the theory developed in \cite{BKM} and this paper is explained in \cite{KK}.
This paper generalizes  \cite{KLM}, where analogous results were obtained for finite type $A$.

 We now give a definition of (graded) affine cellular algebra from \cite[Definition 2.1]{KoXi}. Throughout the paper,  unless otherwise stated, we assume that all algebras are ($\Z$)-graded, all ideals, subspaces, etc. are homogeneous, and all homomorphisms are homogeneous degree zero homomorphisms with respect to the given gradings. For this introduction, we fix a noetherian domain $k$ (later on it will be sufficient to work with $k=\Z$). 
 Let $A$ be a (graded) unital $k$-algebra with a $k$-anti-involution $\tau$. A (two-sided) ideal $J$ in $A$ is called an \emph{affine cell ideal} if the following conditions are satisfied: 
\begin{enumerate}
\item $\tau(J) = J$;
\item there exists an affine $k$-algebra $B$ with a $k$-involution $\si$ and a free $k$-module $V$ of finite rank such that $\Delta:=V \otimes_k B$ has an $A$-$B$-bimodule structure, with the right $B$-module structure induced by the regular right $B$-module structure on $B$;
\item let $\Delta' := B \otimes_k V$ be the $B$-$A$-bimodule with left $B$-module structure induced by the regular left $B$-module structure on $B$ and right $A$-module structure defined by 
\begin{equation}\label{ERightCell}
(b\otimes v)a = \tw(\tau(a)(v \otimes b)),
\end{equation} 
where 
$\tw:V\otimes_k B\to B\otimes_k V,\ v\otimes b\to b\otimes v$; 
then there is an $A$-$A$-bimodule isomorphism $\mu: J \to \Delta \otimes_B\Delta'$, such that the following diagram commutes:

\begin{equation}\label{ECellCD}
\xymatrix{ J \ar^-{\mu}[rr]  \ar^{\tau}[d]&& \Delta \otimes_B\Delta' \ar^{v \otimes b \otimes b' \otimes w \mapsto w \otimes \si(b') \otimes \si(b) \otimes v }[d] \\ J \ar^-{\mu}[rr]&&\Delta \otimes_B\Delta'.}
\end{equation}
\end{enumerate}
The algebra $A$ is called \emph{affine cellular} if there is a $k$-module decomposition $A= J_1' \oplus J_2' \oplus \cdots \oplus J_n'$ 
with $\tau(J_l')=J_l'$ for $1 \leq l \leq n$, such that, setting $J_m:= \bigoplus_{l=1}^m J_l'$, we obtain an ideal filtration
$$0=J_0 \subset J_1 \subset J_2 \subset \cdots \subset J_n=A$$
so that each $J_m/J_{m-1}$ is an affine cell ideal of $A/J_{m-1}$. 

To describe our main results we introduce some notation referring the reader to the main body of the paper for details. Fix a Cartan datum of finite type, and denote by $\Phi_+ = \{\be_1, \dots, \be_N\}$ the set of positive roots, and by $Q_+$ the positive part of the root lattice. For $\al \in Q_+$ we have the KLR algebra $R_\al$ with standard idempotents $\{e(\bi)\mid\bi\in\words_\al\}$. 
We denote by $\Pi(\al)$ the set of root partitions of $\al$. 
This is partially ordered with respect to a certain bilexicographic order `$\leq$'.

To any $\pi\in\Pi(\al)$ one associates a proper standard module $\bar \De(\pi)$ and a word $\bi_\pi\in \words_\al$. We fix a distinguished vector $v_\pi^+ \in \bar \De(\pi)$, and choose a set $\BB_\pi \subseteq R_\al$ so that $\{bv_\pi^+ \mid b \in \BB_\pi\}$ is a basis of $\bar \De(\pi)$. We define polynomial subalgebras $\Sym_\pi\subseteq R_\al$---these are isomorphic to tensor products of algebras of symmetric polynomials. 
We also explicitly define elements $\de_\pi, D_\pi\in R_\al$ and set $e_\pi := D_\pi \de_\pi$. 
Then we set
\begin{align*}
I_\pi' &:= \text{$k$-span}\{b e_\pi \Sym_\pi D_\pi (b')^\tau \mid b,b' \in \BB_\pi\},
\end{align*}
$I_\pi := \sum_{\si \geq \pi}{I_\si'}$, and $I_{>\pi} = \sum_{\si > \pi}{I_\si'}$. 
Our main results are now as follows

\vspace{2mm}
\noindent
{\bf Main Theorem.}
{\em
The algebra $R_\al$ is graded affine cellular with cell chain given by the ideals $\{I_\pi\mid \pi \in \Pi(\al)\}$. Moreover, for a fixed $\pi\in \Pi(\al)$, we set $\bar R_\al := R_\al / I_{>\pi}$ and $\bar h:=h+I_{>\pi}$ for any $h\in R_\al$. We have:
  \begin{enumerate}
  \item[{\rm (i)}] $I_\pi = \sum_{\si \geq \pi} R_\al e(\bi_\si) R_\al$;
   \item[{\rm (ii)}] $\bar e_\pi$ is an idempotent in $\bar R_\al$;
  \item[{\rm (iii)}] the map $\Sym_\pi \to \bar e_\pi \bar R_\al \bar e_\pi,\ f\mapsto \bar e_\pi \bar f\bar e_\pi$ is an isomorphism of graded algebras;
  \item[{\rm (iv)}] $\bar R_\al \bar e_\pi$ is a free right $\bar e_\pi \bar R_\al \bar e_\pi$-module with basis $\{\bar b \bar e_\pi \mid b \in \BB_\pi\}$;
  \item[{\rm (v)}] $\bar e_\pi \bar R_\al $ is a free left $\bar e_\pi \bar R_\al \bar e_\pi$-module with basis $\{\bar e_\pi \bar D_\pi \bar b^\tau \mid b \in \BB_\pi\}$;
  \item[{\rm (vi)}] multiplication provides an isomorphism
      \[
        \bar R_\al \bar e_\pi \otimes_{\bar e_\pi \bar R_\al \bar e_\pi} \bar e_\pi \bar R_\al \iso \bar R_\al \bar e_\pi \bar R_\al;
      \]
\item[{\rm (vii)}]  $\bar R_\al \bar e_\pi \bar R_\al = I_\pi / I_{>\pi}$.
  \end{enumerate}
}

Main Theorem(vii) shows that each affine cell ideal $I_\pi/I_{>\pi}$ in $R_\al/I_{>\pi}$ is generated by an idempotent. This, together with the fact that each algebra $\Sym_\pi$ is a polynomial algebra, is enough to invoke \cite[Theorem 4.4]{KoXi} to get

\vspace{2 mm}
\noindent
{\bf Corollary.}
{\em 
If the ground ring $k$ has finite global dimension, then the algebra $R_\al$ has finite global dimension. 
}

\vspace{2 mm}
This seems to be a slight generalization of \cite{Kato},\cite{McN},\cite{BKM}  in two ways: \cite{Kato} assumes that $k$ is a field of characteristic zero (and the Lie type is simply-laced), while \cite{McN},\cite{BKM} assume that $k$ is a field; moreover, \cite{Kato},\cite{McN},\cite{BKM} deal with categories of graded modules only, while our corollary holds for the algebra $R_\al$ even as an ungraded algebra. 

The paper is organized as follows. Section 2 contains preliminaries needed for the rest of the paper. 
The first subsection contains mostly general conventions that will be used. Subsection~\ref{SLie} goes over the Lie theoretic notation that we employ. We move on in subsection~\ref{SSKLR} to the definition and basic results of Khovanov-Lauda-Rouquier (KLR) algebras. The next two subsections are devoted to recalling results about the representation theory of KLR algebras. Then, in subsection~\ref{SSQG}, we introduce our notation regarding quantum groups, and recall some well-known basis theorems. The next subsection is devoted to the connection between KLR algebras and quantum groups, namely the categorification theorems. Finally, subsection~\ref{SSDF} contains an easy direct proof of a graded dimension formula for the KLR algebras, cf.  \cite[Corollary 3.15]{BKM}. 

Section 3 is devoted to constructing a basis for the KLR algebras that is amenable to checking affine cellularity. We begin in subsection~\ref{SSSSWId} by choosing some special weight idempotents and proving some properties they enjoy. Subsection~\ref{SNota} introduces the notation that allows us to define our affine cellular structure. This subsection also contains the crucial Hypothesis~\ref{HProp}. Next, in subsection~\ref{SSPower}, we come up with an affine cellular basis in the special case corresponding to a root partition of that is a power of a single root. Finally, we use this in the last subsection to come up with our affine cellular basis in full generality.

In section 4 we show how the affine cellular basis is used to prove that the KLR algebras are affine cellular.

Finally, in section 5 we verify Hypothesis~\ref{HProp} for all positive roots in all finite types. We begin in subsection~\ref{SSHomog} by recalling some results concerning homogeneous representations. In subsection~\ref{SSSLO} we recall the definition of special Lyndon orders and Lyndon words, which will serve as the special weights of subsection~\ref{SSSSWId}. The next subsection is devoted to verifying Hypothesis~\ref{HProp} in the special case when the cuspidal representation corresponding to the positive root is homogeneous. We then employ this in subsection~\ref{SSADE} to show that the hypothesis holds in simply-laced types. Finally, we have subsection~\ref{SSBCFG}, wherein we verify the hypothesis by hand in the non-symmetric types.

\section{Preliminaries and a dimension formula} 
In this section we set up the theory of KLR algebras and their connection to  quantum groups following mainly \cite{KL1} and also \cite{KR}. Only subsection~\ref{SSDF} contains some new material.

\subsection{Generalities}
Throughout the paper we work over the ground ring $\O$ which is assumed to be either $\Z$ or an arbitrary field $F$. Most of the time we work over $F$ and then deduce the corresponding result for $\Z$ using the following standard lemma

\begin{Lemma}\label{PFieldToZ}
  Let $M$ be a finitely generated $\Z$-module, and $\{x_\al\}_{\al \in A}$ a subset of $M$. Then $\{x_\al\}$ is a spanning set (resp. basis) of $M$ if and only if $\{1_F \otimes x_\al\}$ is a spanning set (resp. basis) of $F \otimes_\Z M$ for every field $F$.
\end{Lemma}

Let $q$ be an indeterminate, $\Q(q)$ the field of rational functions, and $\Laurent:=\Z[q,q^{-1}]\subseteq \Q(q)$. Let\, $\bar{ }:\Q(q)\to \Q(q)$ be the $\Q$-algebra involution with $\bar q=q^{-1}$, referred to as the {\em bar-involution}. 

For a 
graded vector space $V=\oplus_{n\in \Z} V_n$, with finite dimensional graded components its {\em graded dimension} is $\qdim \, V:=\sum_{n \in \Z}  (\dim V_n)q^n\in\Z[[q,q^{-1}]]$. 
For any graded $F$-algebra $H$ we denote by $\Mod{H}$ the abelian category of all graded left $H$-modules, with 
morphisms being {\em degree-preserving} module homomorphisms, which we denote by $\hom$.
Let $\mod{H}$ denote
the abelian subcategory of all
{\em finite dimensional}\, graded $H$-modules and
 $\Proj{H}$ denote the additive subcategory  of 
all {\em finitely generated projective}\, graded $H$-modules. 
Denote the corresponding Grothendieck groups by
 $[\mod{H}]$ and $[\Proj{H}]$, respectively. 
These Grothendieck groups are $\Laurent$-modules via 
$
q^m[M]:=[M\langle m\rangle],
$ 
where $M\langle m\rangle$ denotes the module obtained by 
shifting the grading up by $m$:
$
M\langle m\rangle_n:=M_{n-m}.
$ 
For $n \in \Z$, let
$
\Hom_H(M, N)_n := \hom_H(M \langle n \rangle, N)
$
denote the space of 
homomorphisms
of degree $n$. 
Set
$
\Hom_H(M,N) := \bigoplus_{n \in \Z} \Hom_H(M,N)_n. 
$


\subsection{Lie theoretic data}\label{SLie}
A {\em Cartan datum} is a pair $(I,\cdot)$ consisting of a set $I$ and a $\Z$-valued symmetric bilinear form $i,j\mapsto i\cdot j$ on the free abelian group $\Z[I]$ such that $i\cdot i\in \{2,4,6,\dots\}$ for all $i\in I$ and 
$2(i\cdot j)/(i\cdot i)\in\{0,-1,-2\dots\}$ for all $i\neq j$ in $I$. 
Set
$
a_{ij}:=2(i\cdot j)/(i\cdot i)$ for $i,j\in I$ and define the {\em Cartan matrix} $A:=(a_{ij})_{i,j\in I}$. 
Throughout the paper, unless otherwise stated, we assume that $A$ has {\em finite type}, see \cite[\S 4]{Kac}. 
We have simple roots $\{\al_i\mid i\in I\}$, and we identify $\al_i$ with $i$. 
Let 
$Q_+ := \bigoplus_{i \in I} \Z_{\geq 0} \al_i$. For $\alpha \in Q_+$, we write $\height(\alpha)$ for the sum of its 
coefficients when expanded in terms of the $\al_i$'s.
Denote by $\Phi_+\subset Q_+$ the set of {\em positive}\, roots, cf. \cite[\S 1.3]{Kac}, and by $W$ the corresponding Weyl group. A total order on $\Phi_+$ is called {\em convex} if $\be,\ga,\be+\ga\in \Phi_+$ and $\be<\ga$ imply $\be<\be+\ga<\ga$.

Given $\beta\in \Z[I]$, denote
\begin{align*}
q_\beta:=q^{(\beta\cdot \beta)/2},\ 
[n]_\beta:=(q_\beta^{n}-q_\beta^{-n})/(q_\beta-q_\beta^{-1}),\ 
 [n]^!_\beta:=[n]_\beta[n-1]_\beta\dots[1]_\beta.
\end{align*}
In particular, for $i\in I$, we have $q_i,[n]_i,[n]_i^!$. 
Let $A$ be a $Q_+$-graded $\Q(q)$-algebra, $\theta\in A_{\al}$ for $\al\in Q_+$, and $n\in\Z_{\geq 0}$. We use the standard notation for  
quantum divided powers:   $\theta^{(n)}:= \theta^n/[n]_\al^!.$

Denote by $\words:=\bigsqcup_{d\geq 0} I^d$ the set of all tuples $\bi=i_1\dots i_d$ of elements of $I$, which we refer to as {\em words}. We consider $\words$ as a monoid under the concatenation product. 
If $\bi\in\words$, we can write it in the form 
$
\bi=j_1^{m_1}\dots j_r^{m_r}
$
for $j_1,\dots,j_r\in I$ such that $j_s\neq j_{s+1}$ for all $s=1,2,\dots,r-1$. 
We then denote
\begin{equation}\label{EIFact}
[\bi]!:=[m_1]^!_{j_1}\dots[m_r]^!_{j_r}. 
\end{equation}
For $\bi=i_1\dots i_d$ set 
$
|\bi|:=\al_{i_1}+\dots+\al_{i_d}\in Q_+.
$ 
The symmetric group $S_d$ 
with simple transpositions $s_1,\dots,s_{d-1}$ 
acts  
on $I^d$ on the left by place
permutations. The $S_d$-orbits on $I^d$ are the sets
$
\words_\alpha := \{\bi\in I^d \:|\:|\bi| = \alpha\}
$ 
parametrized by the elements $\alpha \in Q_+$ of height $d$. 



\subsection{Khovanov-Lauda-Rouquier algebras}\label{SSKLR} 
Let $A$ be a 
Cartan matrix. Choose signs $\eps_{ij}$ for all $i,j \in I$ with $a_{ij}
< 0$  so that $\eps_{ij}\eps_{ji} = -1$, and 
define the polynomials 
$\{Q_{ij}(u,v)\in F[u,v]\mid i,j\in I\}$:    
\begin{equation}\label{EArun}
Q_{ij}(u,v):=
\left\{
\begin{array}{ll}
0 &\hbox{if $i=j$;}\\
1 &\hbox{if $a_{ij}=0$;}\\
\eps_{ij}(u^{-a_{ij}}-v^{-a_{ji}}) &\hbox{if $a_{ij}<0$.}
\end{array}
\right.
\end{equation}
In addition, fix $\al\in Q_+$ of height $d$. 
Let $R_\al=R_\al(\Ga,\O)$ be an associative graded unital $\O$-algebra, given by the generators
\begin{equation*}\label{EKLGens}
\{e(\bi)\mid \bi\in \words_\al\}\cup\{y_1,\dots,y_{d}\}\cup\{\psi_1, \dots,\psi_{d-1}\}
\end{equation*}
and the following relations for all $\bi,\bj\in \words_\al$ and all admissible $r,t$:

\begin{equation}
e(\bi) e(\bj) = \de_{\bi,\bj} e(\bi),
\quad{\textstyle\sum_{\bi \in \words_\alpha}} e(\bi) = 1;\label{R1}
\end{equation}
\begin{equation}\label{R2PsiY}
y_r e(\bi) = e(\bi) y_r;\qquad y_r y_t = y_t y_r;
\end{equation}
\begin{equation}
\psi_r e(\bi) = e(s_r\bi) \psi_r;\label{R2PsiE}
\end{equation}
\begin{equation}\label{R3YPsi}
y_r \psi_s = \psi_s y_r\qquad (r \neq s,s+1);
\end{equation}
\begin{equation}
(y_t\psi_r-\psi_r y_{s_r(t)})e(\bi)  
= \de_{i_r,i_{r+1}}(\de_{t,r+1}-\de_{t,r})e(\bi);
\label{R6}
\end{equation}
\begin{equation}
\psi_r^2e(\bi) = Q_{i_r,i_{r+1}}(y_r,y_{r+1})e(\bi)
 \label{R4}
\end{equation}
\begin{equation} 
\psi_r \psi_t = \psi_t \psi_r\qquad (|r-t|>1);\label{R3Psi}
\end{equation}
\begin{equation}
\begin{split}
&(\psi_{r+1}\psi_{r} \psi_{r+1}-\psi_{r} \psi_{r+1} \psi_{r}) e(\bi)  
\\=
&
\de_{i_r,i_{r+2}}\frac{Q_{i_r,i_{r+1}}(y_{r+2},y_{r+1})-Q_{i_r,i_{r+1}}(y_r,y_{r+1})}{y_{r+2}-y_r}e(\bi).
\end{split}
\label{R7}
\end{equation}
The {\em grading} on $R_\al$ is defined by setting:
$$
\deg(e(\bi))=0,\quad \deg(y_re(\bi))=i_r\cdot i_r,\quad\deg(\psi_r e(\bi))=-i_r\cdot i_{r+1}.
$$
In this paper {\em grading} always means {\em $\Z$-grading}, ideals are assumed to be homogeneous, and modules are assumed graded, unless otherwise stated.

It is pointed out in \cite{KL2} and \cite[\S3.2.4]{R} that up to isomorphism the graded $\O$-algebra $R_\al$ depends only on the Cartan datum and $\al$. 
We refer 
to the algebra $R_\al$ as  an {\em (affine) Khovanov-Lauda-Rouquier algebra}. 
It is convenient to consider the direct sum of algebras 
$
R:=\bigoplus_{\al_\in Q_+} R_\al.
$
Note that $R$ is non-unital, but it is locally unital since each $R_\al$ is unital. 
The algebra $R_\alpha$ possesses a graded anti-automorphism
\begin{equation}\label{star}
\tau:R_\alpha \rightarrow R_\alpha,\ x \mapsto x^\tau
\end{equation}
which is the identity on generators.

For each element $w\in S_d$ fix a reduced expression $w=s_{r_1}\dots s_{r_m}$ and set 
$
\psi_w:=\psi_{r_1}\dots \psi_{r_m}.
$
In general, $\psi_w$ depends on the choice of the reduced expression of $w$. 

\begin{Theorem}\label{TBasis}{\rm \cite[Theorem 2.5]{KL1},  \cite[Theorem 3.7]{R}} 
The following set  is an $\O$-basis of  $R_\al$: 
$ \{\psi_w y_1^{m_1}\dots y_d^{m_d}e(\bi)\mid w\in S_d,\ m_1,\dots,m_d\in\Z_{\geq 0}, \ \bi\in \words_\al\}.
$ 
\end{Theorem}



In view of the theorem, we have a polynomial subalgebra
\begin{equation}\label{EPol}
\Pol_d=\O[y_1,\dots,y_d]\subseteq R_\al.
\end{equation}

Let $\ga_1,\dots,\ga_l$ be elements of $Q_+$ with $\ga_1+\dots+\ga_l=\al$. Then we have a natural embedding 
\begin{equation}\label{EIotaPi}
\iota_{\ga_1,\dots,\ga_l}:R_{\ga_1}\otimes\dots\otimes R_{\ga_l}\into R_\al
\end{equation}
of algebras, whose image is the {\em parabolic subalgebra} $R_{\ga_1,\dots,\ga_l}\subseteq R_\al$.  This is not a unital subalgebra, the image of the identity
element of $R_{\ga_1}\otimes\dots\otimes R_{\ga_l}$ being  
$$
\textstyle1_{\ga_1,\dots,\ga_l} = \sum_{\bi^{(1)} \in \words_{\ga_1},\dots, \bi^{(l)} \in \words_{\ga_l}} e(\bi^{(1)}\dots\bi^{(l)}).
$$

An important special case is where $\al=d\al_i$ is a multiple of  a simple root, in which case  we have that $R_{d\al_i}$ is the 
$d^{th}$ nilHecke algebra $H_d$ 
generated by $\{y_1, \dots, y_d, \psi_1, \dots, \psi_{d-1}\}$ subject to the relations
\begin{align}
  \psi_{r}^2 &= 0
   \label{eq:HeckeRel1}
   \\
  \psi_{r} \psi_{s} &= \psi_{s} \psi_{r} \qquad \textup{if $|r-s| > 1$}
   \label{eq:HeckeRel2}
   \\
  \psi_{r} \psi_{r+1} \psi_{r} &= \psi_{r+1} \psi_{r} \psi_{r+1}
   \label{eq:HeckeRel3}
  \\
  \psi_{r} y_{s} &= y_{s} \psi_{r} \qquad \textup{if $s\neq r,r+1$}
   \label{eq:HeckeRel4}
  \\
  \psi_{r} y_{r+1} &= y_{r} \psi_{r} + 1
   \label{eq:HeckeRel5}
  \\
  y_{r+1} \psi_{r} &= \psi_{r} y_{r} + 1.
   \label{eq:HeckeRel6}
\end{align}
The grading is so that $\deg(y_r) = \al_i\cdot\al_i$ and $\deg(\psi_r) = -\al_i\cdot\al_i$. 
Note that here the elements $\psi_w$ do not depend on a choice of reduced decompositions.

Let $w_0 \in \Si_d$ be the longest element, and define the following elements of $H_d$: 
\[
  \de_d := y_2 y_3^2 \dots y_d^{d-1}, \quad e_d := \psi_{w_0} \de_d.
\]
It is known that 
\begin{equation}\label{EnH}
  e_d \psi_{w_0} = \psi_{w_0},
\end{equation}
and in particular $e_d$ is an idempotent, see for example \cite[\S2.2]{KL1}. The following is a special case of our main theorem for the case where $\al=d\al_i$, which will be used in its proof. It is known that the center $Z(H_d)$ consists of the symmetric polynomials $\O[y_1,\dots,y_d]^{\Si_d}$.

\begin{Theorem}\label{TnilHeckeCellBasis} {\rm \cite[Theorem 4.16]{KLM}}  
  Let $X$ be a $\O$-basis of $\O[y_1,\dots,y_d]^{\Si_d}$ and let $\BB$ be a basis of $\O[y_1,\dots,y_d]$ as an $\O[y_1,\dots,y_d]^{\Si_d}$-module. Then $\{b e_d f \psi_{w_0} (b')^\tau\ |\ b,b' \in \BB, f \in X\}$ is a $\O$-basis of $H_d$. 
\end{Theorem}


\subsection{Basic representation theory of $R_\al$}
By \cite{KL1}, every irreducible graded $R_\al$-module is finite dimensional,  
and  there are finitely many irreducible $R_\al$-modules up to isomorphism and grading shift. For $\bi\in \words_\al$ and $M\in\Mod{R_\al}$, the {\em $\bi$-word space} of $M$ is
$
M_\bi:=e(\bi)M.
$
We have a decomposition of (graded) vector spaces
$
M=\bigoplus_{\bi\in \words_\al}M_\bi.
$
We say that $\bi$ is a {\em word of $M$} if $M_\bi\neq 0$.


We identify in a natural way: 
$$
[\mod{R}]=\bigoplus_{\al\in Q_+}[\mod{R_\al}],\quad [\Proj{R}]=\bigoplus_{\al\in Q_+}[\Proj{R_\al}].
$$



%
Recall the anti-automorphism $\tau$ from (\ref{star}). This allows us to introduce the left $R_\al$-module structure on the graded dual of a finite dimensional $R_\al$-module $M$---the resulting left $R_\al$-module is denoted $M^\circledast$. On the other hand, given any left $R_\al$-module $M$, denote by $M^\tau$ the right $R_\al$-module with the action given by $mx=\tau(x)m$ for $x\in R_\al,m\in M$. Following \cite[(14)]{KL2}, define the {\em Khovanov-Lauda pairing} to be the $\Laurent$-linear pairing 
$$(\cdot,\cdot) : [\Proj{R_\al}]\times[\Proj{R_\al}] \rightarrow \Laurent\cdot
\prod_{i\in I}\prod_{a=1}^{m_i}\frac{1}{(1-q_i^{2a})}
$$ 
such that $([P],[Q]) = \qdim(P^\tau\otimes_{R_\al} Q)$.

Let $\al,\be\in Q_+$. Recalling the isomorphism $\iota_{\al,\be}:R_\al\otimes\ R_\be\to R_{\al,\be}\subseteq R_{\al+\be}$, consider the 
functors 
\begin{align*}
\Ind_{\alpha,\beta}
&:= R_{\alpha+\beta} 1_{\alpha,\beta}
\otimes_{R_{\alpha,\beta}} ?:\Mod{R_{\alpha,\beta}} \rightarrow \Mod{R_{\alpha+\beta}},\\
\Res_{\alpha,\beta}
&:= 1_{\alpha,\beta} R_{\alpha+\beta}
\otimes_{R_{\alpha+\beta}} ?:\Mod{R_{\alpha+\beta}}\rightarrow \Mod{R_{\alpha,\beta}}.
\end{align*}
For $M\in\mod{R_\al}$ and $N\in \mod{R_\beta}$, we denote
$ 
M\circ N:=\Ind_{\al,\be}
(M\boxtimes N).
$ 
The functors of induction define products 
on the Grothendieck groups $[\mod{R}]$ and $[\Proj{R}]$ and  the functors of restriction define coproducts on $[\mod{R}]$ and $[\Proj{R}]$. These products and coproducts make $[\mod{R}]$ and $[\Proj{R}]$ into twisted unital and counital bialgebras \cite[Proposition 3.2]{KL1}.

Let $i\in I$ and $n\in\Z_{> 0}$. 
As explained in \cite[$\S$2.2]{KL1}, 
the algebra
$R_{n \al_i}$ 
has a representation on the polynomials $F[y_1,\dots,y_n]$
such that each $y_r$ acts as multiplication by $y_r$ and each $\psi_r$ acts
as the divided difference operator 
$
\partial_r: f \mapsto \frac{{^{s_r}} f - f}{y_{r}-y_{r+1}}.
$
Let $P(i^{(n)})$ denote this representation of $R_{n\al_i}$
viewed as a 
graded $R_{n\al_i}$-module with grading defined by
$$
\deg(y_1^{m_1} \cdots y_n^{m_n}) := (\al_i\cdot \al_i )(m_1+\cdots+m_n - n(n-1)/4).
$$
By \cite[$\S$2.2]{KL1}, the left regular $R_{n\al_i}$-module decomposes as 
$
P(i^n) \cong [n]^!_i \cdot P(i^{(n)})$. In particular, $P(i^{(n)})$ is projective.
Set
\begin{align*}
\theta_{i}^{(n)}&:= 
\Ind_{\alpha,n \al_i} (? \boxtimes P(i^{(n)})):\Mod{R_\alpha} \rightarrow \Mod{R_{\alpha+n\al_i}},\\
(\theta_{i}^*)^{(n)}&:= 
\Hom_{R'_{n \al_i}}(P(i^{(n)}), ?): \Mod{R_{\alpha+n\al_i}} \rightarrow \Mod{R_{\alpha}},
\end{align*}
where $R'_{n\al_i} := 1 \otimes R_{n\al_i} \subseteq R_{\alpha,n\al_i}$.
These functors induce $\Laurent$-linear maps on the corresponding Grothendieck groups: 
$$\theta_i^{(n)}:[\Proj{R_\alpha}] \rightarrow [\Proj{R_{\alpha+n\al_i}}],\quad
(\theta_{i}^*)^{(n)}:[\mod{R_{\alpha+n\al_i}}] \rightarrow [\mod{R_{\alpha}}].
$$

\subsection{Cuspidal and standard modules} \label{SSSMT}
Standard module theory for $R_\al$ has been developed in \cite{KR,HMM,BKOP,McN}. Here we follow the most general approach of McNamara \cite{McN}. Fix a reduced decomposition $w_0 = s_{i_1} \dots s_{i_N}$ of the longest element $w_0\in  W$. This gives a convex total order on the positive roots $$\Phi_+ = \{ \be_1 > \dots > \be_N\},$$ 
with $\be_{N+1-k} = s_{i_1} \dots s_{i_{k-1}}(\al_{i_k})$. 

To every positive root $\be\in\Phi_+$ of the corresponding root system $\Phi$, one associates a {\em cuspidal module} $L(\be)$. This irreducible module is uniquely determined by the following property: if $\de,\ga\in Q_+$ are non-zero elements such that $\be=\de+\ga$ and $\Res_{\de,\ga}L(\be)\neq 0$, then $\de$ is a sum of positive roots less than $\be$ and $\ga$ is a sum of positive roots greater than $\be$. 



A standard argument involving the Mackey Theorem from \cite{KL1} and convexity as in the proof of \cite[Lemma 2.11]{BKM}, yields:

\begin{Lemma} \label{LBKM} 
Let $\be\in\Phi_+$ and $a_1,\dots,a_n\in\Z_{\geq 0}$. All composition factors of $\Res_{a_1\be,\dots,a_n\be}L(\be)^{\circ (a_1+\dots+a_n)}$ are of the form 
$L(\be)^{\circ a_1}\boxtimes \dots\boxtimes L(\be)^{\circ a_n}$. 
\end{Lemma}

Let $\al\in Q_+$. A tuple $\pi=(p_1, \dots p_N) \in \Z_{\geq 0}^N$ is called a {\em root 
partition of $\al$} if $p_1\be_1+\dots+p_N\be_N=\al$. We also use the notation $\pi=(\be_1^{p_1},\dots,\be_N^{p_N})$. For example, if $\al=n\be$ for $\be\in\Phi_+$, we have a root partition $(\be^n)\in\Pi(\al)$. 
Denote by $\Pi(\al)$ the set of all root partitions of $\al$. 
This set has two total orders: $\leq_l$ and $\leq_r$ defined as follows:  
$(p_1,\dots,p_N)<_l (s_1,\dots,s_N)$ (resp. $(p_1,\dots,p_N)<_r (s_1,\dots,s_N)$) if there exists $1\leq k\leq N$ such that $p_k<s_k$ and $p_m=s_m$ for all $m<k$ (resp. $m>k$). Finally, we have a {\em bilexicographic partial order}:
\begin{equation}\label{EBilex}
\pi\leq \si \Longleftrightarrow \pi\leq_l \si\ \text{and}\ \pi\leq_r\si\qquad(\pi,\si\in\Pi(\al)).
\end{equation}

The following lemma is implicit in \cite{McN}; see also \cite[Lemma 2.5]{BKM}.
\begin{Lemma}\label{LMinRP}
Given any $\pi \in \Pi(p\be)$, we have $\pi \geq (\be^p)$.
\end{Lemma}

For a root partition $\pi=(p_1,\dots,p_N)\in\Pi(\al)$ as above, set $\shift(\pi):=\sum_{k=1}^N (\be_k\cdot \be_k)p_k(p_k-1)/4$, and define the corresponding {\em proper standard module} 
\begin{equation}\label{EStand}
\bar\De(\pi):=L({\be_1})^{\circ p_1}\circ\dots\circ L({\be_N})^{\circ p_N}\langle\shift(\pi)\rangle. 
\end{equation}
For $\pi=(\be_1^{p_1},\dots,\be_N^{p_N})$, we denote
$$
\Res_\pi:=\Res_{p_1\be_1,\dots,p_N\be_N}.
$$

\begin{Theorem} \label{TStand}%
{\rm \cite{McN}} 
For any convex order there exists a cuspidal system $\{L(\be)\mid \be\in \Phi_+\}$.  Moreover: 
\begin{enumerate}
\item[{\rm (i)}] For every $\pi\in\Pi(\al)$, the proper standard module  
$
\bar\De(\pi) 
$ has irreducible head; denote this irreducible module $L(\pi)$. 

\item[{\rm (ii)}] $\{L(\pi)\mid \pi\in \Pi(\al)\}$ is a complete and irredundant system of irreducible $R_\al$-modules up to isomorphism.

\item[{\rm (iii)}] $L(\pi)^\circledast\simeq L(\pi)$.  

\item[{\rm (iv)}] $[\bar\De(\pi):L(\pi)]_q=1$, and $[\bar\De(\pi):L(\si)]_q\neq 0$ implies $\si\leq \pi$. 


\item[{\rm (v)}] $L(\be)^{\circ n}$ is irreducible for every $\be\in \Phi_+$ and every $n\in\Z_{>0}$. 

\item[{\rm (vi)}] $\Res_\pi\bar\De(\si)\neq 0$ implies $\si\geq \pi$, and $\Res_\pi\bar\De(\pi)\simeq L({\be_1})^{\circ p_1}\boxtimes\dots\boxtimes L({\be_N})^{\circ p_N}$. 
\end{enumerate}
\end{Theorem}

Note that the algebra $R_\al(F)$ is defined over $\Z$, i.e. $R_\al(F)\simeq R_\al(\Z)\otimes_\Z F$. We will use the corresponding indices when we need to distinguish between modules defined over different rings. 
The following result shows that cuspidal modules are also defined over $\Z$:

\begin{Lemma} \label{LCuspInt}
Let $\be\in\Phi_+$, and $v\in L(\be)_\Q$ be a non-zero homogeneous vector. Then  $L(\be)_\Z:=R_\be(\Z)\cdot v\subset L(\be)_\Q$ is an $R_\be(\Z)$-invariant lattice such that $L(\be)_\Z\otimes_\Z F\simeq L(\be)_F$ as $R_\be(F)$-modules for any field $F$.   
\end{Lemma}
\begin{proof}
Note using degrees that $L(\be)_\Z$ is finitely generated over $\Z$, hence it is a  lattice in $L(\be)_\Q$. Furthermore $\CH L(\be)_\Z\otimes_\Z F=\CH L(\be)_\Q$, whence by definition of the cuspidal modules, all composition factors of  $\CH L(\be)_\Z\otimes_\Z F$ are of the form $L(\be)_F$. But there is always a multiplicity one composition factor in a reduction modulo $p$ of any irreducible module over a KLR algebra, thank to \cite[Lemma 4.7]{Kcusp}.
\end{proof}

\subsection{Quantum groups}\label{SSQG}

Following \cite[Section 1.2]{Lu1}, we define the algebra $\pf$ to be the free $\Q(q)$-algebra with generators $\ptheta_i$ for $i \in I$ (our $q$ is Lusztig's $v^{-1}$, in keeping with the conventions of \cite{KL1}). This algebra is $Q_+$-graded by assigning the degree $\al_i$ to $\ptheta_i$ for each $i \in I$, so that $\pf = \oplus_{\al \in Q_+} \pf_\al$. If $x \in \pf_\al$, we write $|x| = \al$. For $\bi=(i_1, \dots, i_n) \in \words$, write $\ptheta_{\bi} := \ptheta_{i_1} \dots \ptheta_{i_n}$.
Then 
$ 
  \{\ptheta_\bi \mid \bi \in \words_\al\}
$ 
is a basis for $\pf_\al$. In particular, each $\pf_\al$ is finite dimensional.
Consider the graded dual $\pf^* := \oplus_{\al \in Q_+} (\pf_\al)^*$. We consider words $\bi \in \words$ as elements of $\pf^*$, so that $\bi(\ptheta_\bj) = \de_{\bi, \bj}$. That is to say,
$ 
  \{\bi \mid \bi \in \words_\al\}
$ 
is the basis of $\pf_\al^*$ dual to the basis 
 $\{\ptheta_\bi \mid \bi \in \words_\al\}$.

Let $\pAf$ be the $\A$-subalgebra of $\pf$ generated by $\{(\ptheta_i)^{(n)} \mid i \in I, n \in \Z_{\geq 0} \}$. This algebra is $Q_+$-graded by $\pAf = \oplus_{\al \in Q_+}\pAf_\al$, where $\pAf_\al := \pAf \cap \pf_\al$. Given $\bi = j_1^{r_1}\dots j_m^{r_m} \in \words$ with $j_n \neq j_{n+1}$ for $1 \leq n < m$, denote $\ptheta_{(\bi)} := (\ptheta_{j_1})^{(r_1)} \dots (\ptheta_{j_m})^{(r_m)} \in \pAf$. Then
\begin{equation}\label{EMBasisA}
  \{\ptheta_{(\bi)} \mid \bi \in \words_\al\}
\end{equation}
is an $\A$-basis of $\pAf_\al$.	We also define $\pAf^* := \{x \in \pf^* \mid x(\pAf) \subseteq \A\}$, and assign it the induced $Q_+$-grading. For every $\al \in Q_+$, the $\A$-module $\pAf_\al^*$ is free with basis
\begin{equation}\label{EDualMBasisA}
  \{[\bi]! \bi \mid \bi \in \words_\al\}
\end{equation}
dual to (\ref{EMBasisA}).

There is a {\em twisted} multiplication on $\pf \otimes \pf$ given by $(x \otimes y)(z \otimes w) = q^{-|y|\cdot|z|} xz \otimes yw$ for homogeneous $x,y,z,w \in \pf$. Let $r\colon \pf \to \pf \otimes \pf$ be the algebra homomorphism determined by $r(\ptheta_i) = \ptheta_i \otimes 1 + 1 \otimes \ptheta_i$ for all $i \in I$. By \cite[Proposition 1.2.3]{Lu1} there is a unique symmetric bilinear form $(\cdot, \cdot)$ on $\pf$ such that $(1,1)=1$ and 
\begin{align*}
  (\ptheta_i, \ptheta_j) &= \frac{\de_{i,j}}{1-q_i^2} \quad\text{for } i,j \in I,\\
  (xy, z) &= (x \otimes y, r(z)),\\
  (x, yz) &= (r(x), y \otimes z),
\end{align*}
where the bilinear form on $\pf \otimes \pf$ is given by $(x \otimes x', y \otimes y') = (x,y)(x',y')$.

Define $\f$ to be the quotient of $\pf$ by the radical of $(\cdot,\cdot)$. Denote the image of $\ptheta_i$ in $\f$ by $\theta_i$. The $Q_+$-grading on $\pf$ descends to a $Q_+$-grading on $\f$ with $|\theta_i| = i$. 
Let $\Af$ be the $\A$-subalgebra of $\f$ generated by $\theta_i^{(n)}$ for $i \in I, n \in \Z_{\geq 0}$. This algebra is $Q_+$-graded by $\Af_\al := \Af \cap \f_\al$. Given $\bi = j_1^{r_1}\dots j_m^{r_m} \in \words$ with $j_n \neq j_{n+1}$ for $1 \leq n < m$, denote $\theta_\bi := \theta_{j_1}^{r_1} \dots \theta_{j_m}^{r_m}$ and 
\begin{equation}\label{EThetaI}
\theta_{(\bi)} := \theta_{j_1}^{(r_1)} \dots \theta_{j_m}^{(r_m)} \in \Af.
\end{equation}

We recall the definition of the {\em PBW basis} of $\Af$ from \cite[Part VI]{Lu1}. Recall that a reduced decomposition $w_0 = s_{i_1} \dots s_{i_N}$ yields a total order on the positive roots $\Phi_+ = \{ \be_1 > \dots > \be_N\}$, with $\be_{N+1-k} = s_{i_1} \dots s_{i_{k-1}}(\al_{i_k})$. Now, embed $\Af$ into the upper half of the full quantum group via $\theta_i \mapsto E_i$ and take the braid group generators $T_i := T''_{i,+}$ from \cite[37.1.3]{Lu1}. For $1 \leq k \leq N$, we define
\[
  E_{\be_{N+1-k}} := T_{i_1} \dots T_{i_{k-1}}(\theta_{i_k}) \in \Af_{\be_{N+1-k}}.
\]
For a sequence $\pi = (p_1, \dots p_N) \in \Z_{\geq 0}^N$, we set 
$$E_\pi := E_{\be_1}^{(p_1)} \dots E_{\be_N}^{(p_N)}$$ 
and also define
\begin{equation}\label{ELPi}
l_\pi:=\prod_{r=1}^N \prod_{s=1}^{p_k} \frac{1}{1-q_{\be_r}^{2s}}.
\end{equation}

The next theorem now gives a PBW basis of $\Af_\al$.

\begin{Theorem}\label{TPBW}
  The set
$
  \{E_\pi \mid \pi \in \Pi(\al)\}
$ 
is an $\A$-basis of $\Af_\al$. Furthermore: 
\[
  (E_\pi, E_\si) = \de_{\pi, \si} l_\pi.
\]
\end{Theorem}
\begin{proof}
This follows from  Corollary 41.1.4(b), Propositions 41.1.7, 38.2.3, and Lemma 1.4.4 of Lusztig \cite{Lu1}.
\end{proof}

Consider the graded dual $\f^* := \oplus_{\al \in Q_+} \f_\al^*$. The map $r^*: \f^* \otimes \f^* \to \f^*$ gives $\f^*$ the structure of an associative algebra. Let 
\begin{equation}\label{EIota}
\kappa: \f^* \into \pf^*
\end{equation}
be the map dual to the quotient map $\xi:\pf \onto \f$. Set $\Af^* := \{x \in \f^* \mid x(\Af) \subseteq \A\}$ with the induced $Q_+$-grading. Given $i \in I$, we denote by $\theta_i^*: \f^* \to \f^*$ the dual map to the map $\f \to \f,\ x \mapsto x \theta_i$. Then the divided power $(\theta_i^*)^{(n)}:\f^*\to\f^*$ is dual to the map $x \mapsto x \theta_i^{(n)}$. Clearly $(\theta_i^*)^{(n)}$ stabilizes $\Af^*$. For $\be \in \Phi_+$, define $E_\be^* \in \Af_\be^*$ to be dual to $E_\be$. We define
\begin{equation}\label{EDualPBW}
  (E_\be^*)^{\< m\>} := q_\be^{m(m-1)/2} (E_\be^*)^m \quad\text{and}\quad E_\pi^* := (E_{\be_1}^*)^{\< p_1 \>} \dots (E_{\be_N}^*)^{\< p_N\>}
\end{equation}
for $m \geq 0$, and any sequence $\pi = (p_1, \dots p_N) \in \Z_{\geq 0}^N$. The next well-known result gives the {\em dual PBW basis} of $\Af^*$.

\begin{Theorem}\label{TDualPBW}
  The set
$
  \{E_\pi^* \mid \pi \in \Pi(\al)\}
$ 
is the $\A$-basis of $\Af_\al^*$ dual to the PBW basis of Theorem~\ref{TPBW}.
\end{Theorem}
\begin{proof}
It easily follows from the properties of the Lusztig bilinear form, the definition of the product on $\f^*$ and Theorem~\ref{TPBW} that the linear functions 
$(E_\pi,-)$ and $l_\pi E_\pi^*$ on $\f$ are equal. It remains to apply Theorem~\ref{TPBW} one more time. 
\end{proof}

\begin{Example} 
{\rm 
Let $C=A_2$, and $w_0=s_1s_2s_1$. Then $E_{\al_1+\al_2}=T_{1,+}''(E_2)=E_1E_2-qE_2E_1$, and, switching back to $\theta$'s, the PBW basis of $\Af_{\al_1+\al_2}$ is $\{\theta_2\theta_1, \theta_1\theta_2-q\theta_2\theta_1\}$. Using the defining properties of Lusztig's bilinear form, one can easily check that $(E_{\al_2}E_{\al_1},E_{\al_2}E_{\al_1})=\frac{1}{(1-q^2)^2}$, $(E_{\al_1+\al_2},E_{\al_1+\al_2})=\frac{1}{(1-q^2)}$, and $(E_{\al_2}E_{\al_1},E_{\al_1+\al_2})=0$. Finally the dual basis is  $\{(12),(21)+q(12)\}$. 
}
\end{Example}

\subsection{Categorification of $\Af$ and $\Af^*$}
Now we state the fundamental categorification theorem 
proved in \cite{KL1,KL2}, see also \cite{R}. We denote by $[R_0]$ the class of the left regular representation of the trivial algebra $R_0\cong F$. 

\begin{Theorem}\label{klthm}
There is a unique $\Laurent$-linear isomorphism
$
\gamma:\Af \stackrel{\sim}{\rightarrow} [\Proj{R}]
$
such that 
$1 \mapsto [R_0]$ 
and 
$\gamma( x \theta_i^{(n)}) = 
\theta_i^{(n)}(\gamma(x))$ 
for all $x \in \Af$,  $i \in I$, and $n \geq 1$. 
Under this isomorphism:
\begin{itemize}
\item[(1)] $\gamma(\Af_\al)=[\Proj{R_\al}]$;
\item[(2)] the 
multiplication $\Af_\alpha \otimes \Af_\beta
\rightarrow  \Af_{\alpha+\beta}$ 
corresponds to the product on $[\Proj{R}]$ induced by the exact functor 
$\Ind_{\alpha,\beta}$;
\item[(3)] for $\bi \in \words_\al$ we have $\gamma(\theta_\bi) = [R_\al e(\bi)]$;
\item[(4)] for $x,y \in \Af$ we have $(x,y) = (\ga(x),\ga(y))$.
\end{itemize}
\end{Theorem}


Let $M$ be a finite dimensional graded $R_\al$-module. Define the {\em $q$-character} of $M$ as follows: 
\begin{equation*}
\CH M:=\sum_{\bi\in \words_\al}(\qdim M_\bi) \bi\in \pAf^*.
\end{equation*}
The $q$-character map $\CH: \mod{R_\al}\to \pAf^*$ factors through to give an $ \Laurent$-linear map from the Grothendieck group
\begin{equation}\label{EChMap}
\CH: [\mod{R_\al}]\to  \pAf^*.
\end{equation}

We now state a dual result to Theorem~\ref{klthm}, see \cite[Theorem 4.4]{KR}.

\begin{Theorem}\label{Dklthm} 
There is a unique $\Laurent$-linear isomorphism
$
\gamma^*: [\mod{R}]\stackrel{\sim}{\rightarrow}\Af^*
$
with the following properties:
\begin{itemize}
\item[(1)] $\ga^*([R_0])= 1$; 
\item[(2)] $\gamma^*((\theta_i^*)^{(n)}(x)) = 
(\theta_i^*)^{(n)}(\gamma^*(x))$ for all $x \in [\mod{R}],\ 
i \in I,\ n \geq 1$; 
\item[(3)] the following triangle is commutative:
$$
\begin{pb-diagram}
\node{}\node{\pAf^*}
\node{} \\
\node{[\mod{R}]} \arrow[2]{e,t}{\ga^*}
\arrow{ne,t}{\CH}
\node{}\node{\Af^*}
\arrow{nw,t}{\kappa}
\end{pb-diagram}
$$
\item[(4)] $\gamma^*([\mod{R_\al}])=\Af^*_\al$ for all $\ga\in Q_+$;
\item[(5)] under the isomorphism $\ga^*$, the 
multiplication $\Af^*_\alpha \otimes \Af^*_\beta
\rightarrow  \Af^*_{\alpha+\beta}$  
corresponds to the product on $[\mod{R}]$ induced by 
$\Ind_{\alpha,\beta}$;
\end{itemize}
\end{Theorem}

We conclude with McNamara's result on the categorification of the dual PBW-basis (see also \cite{Kato} for simply laced Lie types): 

\begin{Lemma}\label{LStdPBW}
  For every $\pi \in \Pi(\al)$ we have $\ga^*([\bar \De(\pi)]) = E_\pi^*$.
\end{Lemma}
\begin{proof}
  By \cite[Theorem 3.1(1)]{McN}, we have $\ga^*([L(\be)]) = E_\be^*$ for all $\be\in\Phi_+$.  
The general case then follows from Theorem~\ref{Dklthm}(5) and the definition (\ref{EStand}) of $\bar\De(\pi)$. 
\end{proof}

\subsection{A dimension formula}\label{SSDF}
In this section we obtain a dimension formula for $R_\al$, which can be viewed as a combinatorial shadow of the affine quasi-hereditary structure on it. The idea of the proof comes from \cite[Theorem 4.20]{BKgrdec}. An independent but much less elementary proof can be found in \cite[Corollary 3.15]{BKM}. 

Recall the element $\theta_{(\bi)}\in \Af$ from (\ref{EThetaI}) and the scalar $[\bi]!\in\A$ from (\ref{EIFact}). We note that Lemma~\ref{LMonPBWGen} and Theorem~\ref{TDimGen} {\em do not require the assumption that the Cartan matrix $A$ is of finite type}, adopted elsewhere in the paper. 

\begin{Lemma}\label{LMonPBWGen}
Let $V^1,\dots,V^m\in \mod{R_\al}$, and let $v^n:=\ga^*([V^n])\in\Af^*_\al$ for $n=1,\dots,m$. Assume that $\{v^1,\dots,v^m\}$ is an $\A$-basis of $\Af^*_\al$. Let $\{v_1,\dots,v_m\}$ be the dual basis of $\Af_\al$. 
 Then for every $\bi \in \words_\al$, we have 
\[
    \theta_{(\bi)} = \sum_{n=1}^m\frac{\qdim  V^n_\bi}{[\bi]!} v_n.
\]
\end{Lemma}
\begin{proof}
Recall the map $\kappa$ from (\ref{EIota}) dual to the natural projection $\xi:\pf \onto \f$. By Theorem~\ref{Dklthm} we have for any $1\leq n\leq m$: 
\begin{align*}
  \kappa(v^n) =  \kappa(\ga^*([V^n]))=\CH([V^n])=\sum_{\bi \in \words_\al}(\qdim V^n_\bi) \bi 
= \sum_{\bi \in \words_\al}\frac{\qdim V^n_\bi}{[\bi]!} [\bi]! \bi.
\end{align*}
Recalling (\ref{EMBasisA}) and (\ref{EDualMBasisA}), 
$\{\ptheta_{(\bi)} \mid \bi \in \words_\al\}$ and
$\{[\bi]! \bi \mid \bi \in \words_\al\}$ is a pair of dual bases in $\pAf_\al $ and $\pAf_\al ^*$. So, using our expression for $ \kappa(v^n)$, we can now get by dualizing:
\begin{align*}
\theta_{(\bi)}&=\xi(\ptheta_{(\bi)})=\sum_{n=1}^m v^n(\xi(\ptheta_{(\bi)}))v_n
\\
&=\sum_{n=1}^m\kappa(v^n)(\ptheta_{(\bi)})v_n
=\sum_{n=1}^m
\sum_{\bj \in \words_\al}\frac{\qdim V^n_\bj}{[\bj]!} [\bj]! \bj
(\ptheta_{(\bi)})v_n
\\
&= \sum_{n=1}^m
\sum_{\bj \in \words_\al}\frac{\qdim V^n_\bj}{[\bj]!} \de_{\bi,\bj}v_n
= \sum_{n=1}^m\frac{\qdim  V^n_\bi}{[\bi]!} v_n,
\end{align*}
as required. 
\end{proof}

\begin{Theorem}\label{TDimGen}
With the assumptions of Lemma~\ref{LMonPBWGen}, for every $\bi, \bj \in \words_\al$, we have
\[
    \qdim(e(\bi)R_\al e(\bj)) = \sum_{n,k=1}^m (\qdim V^n_\bi)(\qdim V^k_\bj)(v_n,v_k).
\]
  In particular,
\[
    \qdim(R_\al) = \sum_{n,k=1}^m (\qdim V^n)(\qdim V^k)(v_n,v_k).
\]
\end{Theorem}
\begin{proof}
  Theorem~\ref{klthm}(3) shows that $[R_\al e(\bi)] = \ga(\theta_\bi) = \ga([\bi]! \theta_{(\bi)})$. Using the definitions and Theorem~\ref{klthm}(4), we have
\begin{align*}
  \qdim(e(\bi)R_\al e(\bj)) &= \qdim((R_\al e(\bi))^\tau \otimes_{R_\al} R_\al e(\bj))\\
  &= ([R_\al e(\bi)], [R_\al e(\bj)])= ([\bi]! \theta_{(\bi)}, [\bj]! \theta_{(\bj)}).
\end{align*}
Now, by Lemma~\ref{LMonPBWGen} we see that
\begin{align*}
  ([\bi]! \theta_{(\bi)}, [\bj]! \theta_{(\bj)}) &= \Big( \sum_{n=1}^m(\qdim V^n_\bi) v_n, \sum_{k=1}^n(\qdim V^k_\bj) v_k \Big),
\end{align*}
which implies the theorem.
\end{proof}

Recall the scalar $l_\pi$ from (\ref{ELPi}), the module $\bar\De(\pi)$ from (\ref{EStand}), and PBW-basis elements $E_\pi$ from \S\ref{SSQG}.

\begin{Corollary}\label{TDim}
  For every $\bi, \bj \in \words_\al$, we have
\[
    \qdim(e(\bi)R_\al e(\bj)) = \sum_{\pi \in \Pi(\al)} (\qdim \bar \De(\pi)_\bi)(\qdim \bar \De(\pi)_\bj)
               l_\pi.
\]
  In particular,
\[
    \qdim(R_\al) = \sum_{\pi \in \Pi(\al)} (\qdim \bar \De(\pi))^2
                   l_\pi.
\]
\end{Corollary}
\begin{proof}
By Lemma~\ref{LStdPBW}, we have $\ga^*(\bar\De(\pi))=E^*_\pi$ for all $\pi\in\Pi(\al)$. Moreover, by Theorem~\ref{TDualPBW}, $\{E_\pi^*\mid \pi\in \Pi(\al)\}$ and $\{E_\pi\mid \pi\in \Pi(\al)\}$ is a pair of dual bases in $\Af_\al^*$ and $\Af_\al$. Finally,   $(E_\pi,E_\si)=\de_{\pi,\si}l_\pi$ by Theorem~\ref{TPBW}. It remains to apply Theorem~\ref{TDimGen}. 
\end{proof}

\section{Affine cellular structure}

Throughout this section we fix $\al\in Q_+$ and a total order $\leq$ on the set $\Pi(\al)$ of root partitions of $\al$, which refines the bilexicographic partial order (\ref{EBilex}).

\subsection{Some special word idempotents}\label{SSSSWId}
Recall from Section~\ref{SSSMT} that for each $\beta\in \Phi_+$, we have a cuspidal module $L(\be)$. Every irreducible $R_\al$-module $L$ has a word space $L_\bi$ such that the lowest degree component of $L_\bi$ is one-dimensional, see for example \cite[Lemma 2.30]{Kcusp} or \cite[Lemma 4.5]{BKM} for two natural choices. 
From now on, for each $\be\in\Phi_+$ we make an arbitrary choice of such word $\bi_\be$ for the cuspidal module $L(\be)$. 

For $\pi=(\be_1^{p_1},\dots,\be_N^{p_N})\in \Pi(\al)$, define
\begin{align*}
\bi_\pi&:=\bi_{\be_1}^{p_1}\dots \bi_{\be_N}^{p_N},
\\
 I_\pi &:= \sum_{\si \geq \pi} R_\al e(\bi_\si) R_\al,
 \\
 I_{>\pi} &:= \sum_{\si > \pi} R_\al e(\bi_\si) R_\al,
\end{align*}
 the sums being over $\si \in \Pi(\al)$. 
 We also consider the (non-unital) embedding of algebras:
\[
  \iota_\pi := \iota_{p_1\be_1,\dots,p_N\be_N} : R_{p_1 \be_1} \otimes \dots \otimes R_{p_N \be_N} \into R_\al, 
\]
whose image is the parabolic subalgebra
$$
R_\pi := R_{p_1\be_1,\dots,p_N\be_N}. 
$$

\begin{Lemma} \label{LAllIdem}
If a two-sided ideal $J$ of $R_\al$ contains all idempotents $e(\bi_\pi)$ with $\pi\in\Pi(\al)$, then $J=R_\al$.
\end{Lemma}
\begin{proof}
If $J\neq R_\al$, let $I$ be a maximal left ideal containing $J$. Then $R_\al / I \cong L(\pi)$ for some $\pi$. Then $e(\bi_\pi)L(\pi)\neq 0$, which contradicts the assumption that $e(\bi_\pi)\in J$. This argument proves the lemma over any field, and then it also follows for $\Z$. 
\end{proof}

\begin{Lemma}\label{LBadWordsNew}
  Let $\pi\in\Pi(\al)$ and $e \in R_\al$ a homogeneous idempotent. If $e L(\si) = 0$ for all $\si\leq \pi$, then $e \in I_{>\pi}$.
\end{Lemma}
\begin{proof}
  Let $I$ be any maximal (graded) left ideal containing $I_{>\pi}$. Then $R_{\al} / I \cong L(\si)$ for some $\si \in \Pi(\al)$ such that $\si\leq \pi$. Indeed, if we had $\si > \pi$ then by definition $e(\bi_\si) \in I_{>\pi} \subseteq I$, and so $e(\bi_\si) L(\si) = e(\bi_\si)(R_\al/I)=0$, which is a contradiction. 
  
We have shown that $e$ is contained in every maximal left ideal containing $I_{>\pi}$. By a standard argument, explained in \cite[Lemma 5.8]{KLM}, we conclude that $e \in I_{>\pi}$. 
\end{proof}

\begin{Corollary}\label{LBadWords}
  Suppose that $\al = p\be$ for some $p \geq 1$ and $\be \in \Phi_+$. Let $\bi \in \words_{\al}$. If $e(\bi) L(\be^p) = 0$, then $e(\bi) \in I_{>(\be^p)}$.
\end{Corollary}
\begin{proof}
  This follows from Lemma~\ref{LMinRP} together with Proposition~\ref{LBadWordsNew}.
\end{proof}


\begin{Lemma}\label{LIotaImage}
  Let $\pi = (\be_1^{p_1}\dots \be_N^{p_N})\in \Pi(\al)$. Then $R_\pi \subseteq I_\pi$.
\end{Lemma}
\begin{proof}
By Lemma~\ref{LAllIdem}, we have
$$
R_{p_n\be_n}=\sum_{\pi^{(n)}\in\Pi(p_n\be_n)}R_{p_n\be_n}e(\bi_{\pi^{(n)}})R_{p_n\be_n}
$$
for all $n=1,\dots,N$. Therefore the image of $\iota_\pi$ equals 
\begin{equation}\label{EIotaImage}
  \sum R_\pi e(\bi_{\pi^{(1)}}\dots \bi_{\pi^{(N)}}) R_\pi
\end{equation}
where the sum is over all $\pi^{(1)}\in\Pi(p_1\be_1),\dots, \pi^{(N)}\in\Pi(p_N\be_N)$.   Fix $\pi^{(n)}\in\Pi(p_n\be_n)$ for all $n=1,\dots,N$. If $\pi^{(n)} = (\be_n^{p_n})$ for every $n$, then $\bi_{\pi^{(1)}}\dots \bi_{\pi^{(N)}} = \bi_\pi$, and the corresponding term of (\ref{EIotaImage}) is in $I_\pi$ by definition.

Let us now assume that $\pi^{(k)} \neq (\be_k^{p_k})$ for some $k$. In view of Lemma~\ref{LMinRP}, we have $\pi^{(k)} > (\be_k^{p_k})$. For any $\si \in \Pi(\al)$ we have $e(\bi_{\pi^{(1)}}\dots \bi_{\pi^{(N)}})\bar\De(\si) \subseteq \Res_\pi \bar\De(\si)$, and by Theorem~\ref{TStand}(vi), if $\si < \pi$ then $\Res_\pi \bar\De(\si)=0$. Furthermore, for $\si = \pi$ we have, by Theorem~\ref{TStand}(vi) applied again  
\begin{align*}
e(\bi_{\pi^{(1)}}\dots \bi_{\pi^{(N)}})\bar\De(\pi) &= e(\bi_{\pi^{(1)}}\dots \bi_{\pi^{(N)}}) \Res_\pi \bar \De(\pi)\\
 &=  e(\bi_{\pi^{(1)}}\dots \bi_{\pi^{(N)}})(\bar\De(\be_1^{p_1}) \boxtimes \dots \boxtimes \bar\De(\be_N^{p_N}))\\
 &\subseteq (\Res_{\pi^{(1)}} \bar\De(\be_1^{p_1})) \boxtimes \dots \boxtimes (\Res_{\pi^{(N)}} \bar\De(\be_N^{p_N})),
\end{align*}
which is zero since $\Res_{\pi^{(k)}} \bar\De(\be_k^{p_k}) = 0$ by Theorem~\ref{TStand}(vi) again. We have shown that for all $\si \leq \pi$ we have $e(\bi_{\pi^{(1)}}\dots \bi_{\pi^{(N)}})\bar\De(\si)=0$, and consequently $e(\bi_{\pi^{(1)}}\dots \bi_{\pi^{(N)}}) L(\si)=0$. Applying Lemma~\ref{LBadWordsNew}, we have that $e(\bi_{\pi^{(1)}}\dots \bi_{\pi^{(N)}}) \in I_{>\pi} \subseteq I_\pi$.
\end{proof}

The following result will often allow us to reduce to the case of a smaller height. 

\begin{Proposition}\label{LParaIdeal1New}
Let  $\ga_1,\dots,\ga_m\in Q_{+}$, $1\leq k\leq m$, and $\pi_0\in\Pi(\ga_k)$. Assume that $\pi\in\Pi(\ga_1+\dots+\ga_m)$ is such that all idempotents from the set
$$
E=\{e(\bi_{\pi^{(1)}}\dots \bi_{\pi^{(m)}})\mid \pi^{(n)}\in\Pi(\ga_n)\ \text{for all $n=1,\dots,m$ and $\pi^{(k)}>\pi_0$}\}
$$
annihilate the irreducible modules $L(\si)$ for all $\si\leq \pi$. 
Then  
\begin{equation}\label{EIotaOfR}
  \iota_{\ga_1,\dots,\ga_m}(R_{\ga_1} \otimes\dots\otimes  R_{\ga_{k-1}}\otimes I_{>\pi_0}\otimes R_{\ga_{k+1}} \otimes\dots\otimes  R_{\ga_m}) \subseteq I_{>\pi}.
\end{equation}
\end{Proposition}
\begin{proof}
We may assume that $\ga_k\neq 0$ since otherwise $I_{>\pi_0}=0$, and the result is clear. By Lemma~\ref{LAllIdem}, we have
$
R_{\ga_n}=\sum_{\pi^{(n)}\in\Pi(\ga_n)}R_{\ga_n}e(\bi_{\pi^{(n)}})R_{\ga_n}
$
for all $n=1,\dots,m$, and by definition, we have
$
I_{>\pi_0}=\sum_{\pi^{(k)}>\pi_0}R_{\ga_k}e(\bi_{\pi^{(k)}})R_{\ga_k}.
$
Therefore the left hand side of (\ref{EIotaOfR}) equals 
$
  \sum_{e\in E} R_{\ga_1,\dots,\ga_m} e R_{\ga_1,\dots,\ga_m}.
$
The result now follows by applying Lemma~\ref{LBadWordsNew}. 
\end{proof}

Recall from Lemma~\ref{LIotaImage} that $\im (\iota_\pi)\subseteq I_\pi$. 

\begin{Corollary}\label{LParaIdeal2}
Let $\pi = (\beta_1^{p_1}, \dots, \beta_N^{p_N})\in\Pi(\al)$ and $1\leq k\leq N$. 
Then 
\[
  \iota_\pi(R_{p_1 \be_1} \otimes \dots \otimes R_{p_{k-1}\be_{k-1}} \otimes I_{>(\be_k^{p_k})} \otimes R_{p_{k+1} \be_{k+1}} \otimes \dots \otimes R_{p_N \be_N}) \subseteq I_{>\pi}.
\]
In particular, the composite map $R_\pi \stackrel{\iota_\pi}{\longrightarrow}I_\pi\longrightarrow I_\pi / I_{>\pi}$ factors through the quotient $R_{p_1 \be_1} / I_{>(\be_1^{p_1})} \otimes \dots \otimes R_{p_N \be_N} / I_{>(\be_N^{p_N})}$. 
\end{Corollary}
\begin{proof}
Apply Proposition~\ref{LParaIdeal1New} with $m=N$, $\ga_n=p_n\be_n$, for $1\leq n \leq N$, $\pi_0=(\be_k^{p_k})$, and $\pi=\pi$. We have to prove that any $e=
e(\bi_{\pi^{(1)}}\dots \bi_{\pi^{(N)}})
\in E$ annihilates all $L(\si)$ for $\si\leq \pi$. We prove more, namely that $e$ annihilates $\bar\De(\si)$ for all $\si\leq\pi$. By Theorem~\ref{TStand}(vi):
\begin{align*}
e\bar\De(\si)&=e\Res_\pi \bar\De(\si)=e\de_{\pi,\si}(L({\be_1})^{\circ p_1}\boxtimes\dots\boxtimes L({\be_1})^{\circ p_1})\\
&=\de_{\pi,\si}\,e(\bi_{\pi^{(1)}})L({\be_1})^{\circ p_1}\boxtimes \dots\boxtimes  e(\bi_{\pi^{(N)}})L({\be_N})^{\circ p_N},
\end{align*}
which is zero since 
$$
e(\bi_{\pi^{(k)}})L({\be_k})^{\circ p_k}=e(\bi_{\pi^{(k)}})\Res_{\pi^{(k)}}L({\be_k})^{\circ p_k}=0
$$
by Theorem~\ref{TStand}(vi) again. 
\end{proof}

\begin{Corollary}\label{LParaIdeal1}
For $\be\in \Phi_+$ and $a, b, c \in \Z_{\geq 0}$ 
we have 
\[
  \iota_{a\be,b\be,c\be}(R_{a \be} \otimes I_{>(\be^b)} \otimes R_{c\be}) \subseteq I_{>(\be^{a+b+c})}.
\]
\end{Corollary}
\begin{proof}
We apply Proposition~\ref{LParaIdeal1New} with $m=3$, $k=2$, $\ga_1=a\be$, $\ga_2=b\be$, $\ga_3=c\be$, $\pi_0=(\be^b)$ and $\pi=(\be^{a+b+c})$. Pick an idempotent $e=e(\bi_{\pi^{(1)}} \bi_{\pi^{(2)}} \bi_{\pi^{(3)}})\in E$. Since $\pi$ is the minimal element of $\Pi((a+b+c)\be)$, it suffices to prove that $eL(\pi)=0$. Note that 
$eL(\pi)=e\Res_{a\be,b\be,c\be}L(\pi)$, so using Lemma~\ref{LBKM}, we just need to show that $e(L(\be)^{\circ a}\boxtimes L(\be)^{\circ b}\boxtimes L(\be)^{\circ c})=0$. But
$$
e(L(\be)^{\circ a}\boxtimes L(\be)^{\circ b}\boxtimes L(\be)^{\circ c})=e(\bi_{\pi^{(1)}})L(\be)^{\circ a}\boxtimes e(\bi_{\pi^{(2)}})L(\be)^{\circ b}\boxtimes e(\bi_{\pi^{(3)}})L(\be)^{\circ c}
$$
is zero, since 
$e(\bi_{\pi^{(2)}})L(\be)^{\circ b}=e(\bi_{\pi^{(2)}})\Res_{\pi^{(2)}}L(\be)^{\circ b}=0
$
 by Theorem~\ref{TStand}(vi).  
\end{proof}

Repeated application of Corollary~\ref{LParaIdeal1} gives the following result.
\begin{Corollary}\label{COneRootIdeal}
For $\be\in \Phi_+$ and $p \in \Z_{> 0}$ we have 
\[
  \iota_{\be,\dots,\be}(R_{\be} \otimes \dots \otimes I_{>(\be)} \otimes \dots \otimes R_{\be}) \subseteq I_{>(\be^p)}.
\]
\end{Corollary}

\subsection{Basic notation concerning cellular bases}\label{SNota}
Let $\be$ be a fixed  positive root of height $d$. 
Recall that we have made a choice of $\bi_\be$ so that in the word space $e(\bi_\be) L(\be)$ of the cuspidal module, the lowest degree part is $1$-dimensional. We fix its spanning vector $v_\be^-$ defined over $\Z$, see Lemma~\ref{LCuspInt}. Similarly, the highest degree part is spanned over $\Z$ by some $v_\be^+$.

We consider the element of the symmetric group $w_{\be,r}\in \Si_{pd}$ 
$$
w_{\be,r}:=\prod_{k=1}^d((r-1)d+k,rd+k).
$$
which permutes the $r$th and the $(r+1)$st `$d$-blocks'. Now define
$$
\psi_{\be,r}:=\psi_{w_{\be,r}}\in R_{p\be}.
$$
Moreover, for $u\in\Si_p$ with a fixed reduced decomposition $u=s_{r_1}\dots s_{r_m}$,  define the elements 
\begin{align*}
w_{\be,u}&:=w_{\be,r_1}\dots w_{\be,r_m}\in\Si_{pd},
\\
\psi_{\be,u}&:=\psi_{\be,r_1}\dots\psi_{\be,r_m}\in R_{p\be}.
\end{align*}

In Section~\ref{SVerif}, we will explicitly define homogeneous elements 
$$
\de_\be,\ D_\be,\ y_\be\ \in\ e(\bi_\be)R_\be e(\bi_\be)
$$
and $e_\be := D_\be \de_\be$ so that the following hypothesis is satisfied: 

\begin{Hypothesis}\label{HProp}
 We have: 
  \begin{enumerate}
    \item[\textrm{(i)}] $e_\be^2 - e_\be \in I_{>(\be)}$.
    \item[\textrm{(ii)}] $\de_\be, D_\be$ and $y_\be$ are $\tau$-invariant.
    \item[\textrm{(iii)}] 
    $\de_\be v_\be^- = v_\be^+$ and $D_\be v_\be^+ = v_\be^-$,
    \item[\textrm{(iv)}] $y_\be$ has degree $\be\cdot \be$ and commutes with $\de_\be$ and $D_\be$,
    \item[\textrm{(v)}] 
The algebra  $(e_\be R_\be e_\be + I_{>(\be)})/I_{>(\be)}$ is generated by  $e_\be y_\be e_\be + I_{>(\be)}$.
    \item[\textrm{(vi)}] $\iota_{\be,\be}(D_\be \otimes D_\be) \psi_{\be,1} = \psi_{\be,1} \iota_{\be,\be}(D_\be \otimes D_\be)$.
  \end{enumerate}
\end{Hypothesis}

From now on until we verify it in Section~\ref{SVerif}, we will work under the assumption that Hypothesis~\ref{HProp} holds. It turns out that this hypothesis is sufficient to construct affine cellular bases.

\begin{Lemma}\label{CProp}
  $R_\be e_\be R_\be + I_{>(\be)} = R_\be$
\end{Lemma}
\begin{proof}
This follows as in the proof of Lemma~\ref{LAllIdem} using $e_\be L(\be)\neq 0$.
\end{proof}

Using Lemma~\ref{LCuspInt}, we can choose a set 
$$
\BB_\be\subseteq R_\be 
$$
of elements defined over $\Z$  such that 
$$
\{b v_\be^-\mid b\in\BB_\be\}
$$
is an $\O$-basis of $L(\be)_\O$.

Fix $p\in\Z_{>0}$ and define the set
\begin{align*}
\BB_{\be^{\boxtimes p}}:=\{\iota_{\be,\dots,\be}(b_1\otimes\dots\otimes b_p)\mid b_1,\dots,b_p\in\BB_\be\},
\end{align*}
and the element 
$$
y_{\be,r}:=\iota_{(r-1)\be,\be,(p-r)\be}(1\otimes y_\be\otimes 1)\in R_{p\be}
\qquad(1\leq r\leq p). 
$$

Further, define the elements of $R_{p\be}$
\begin{align}
e_{\be^{\boxtimes p}}&:=\iota_{\be,\dots,\be}(e_\be,\dots,e_\be),\\
\de_{(\be^p)}&:=y_{\be,2}y_{\be,3}^2\dots y_{\be,p}^{p-1} \iota_{\be,\dots,\be}(\de_\be \otimes \dots \otimes \de_\be),\\
D_{(\be^{p})}&:=\psi_{\be,w_0} \iota_{\be,\dots,\be}(D_\be \otimes \dots \otimes D_\be),\\
e_{(\be^p)}&:=D_{(\be^p)} \de_{(\be^p)} = \psi_{\be,w_0}y_{\be,2}y_{\be,3}^2\dots y_{\be,p}^{p-1} e_{\be^{\boxtimes p}},\label{EDeSi}
\end{align}
where $w_0 \in \Si_p$ is the longest element. It will be proved in Corollary~\ref{C030413_3} that $e_{(\be^p)}^2-e_{(\be^p)}\in I_{>(\be^p)}$ generalizing part (i) of Hypothesis~\ref{HProp}. 
It is easy to see, as in \cite[Lemma~2.4]{KLM}, that there is always a choice of a reduced decompositon of $w_0$ such that 
\begin{equation}\label{ETauInv1}
\psi_{\be,w_0}^\tau=\psi_{\be,w_0}.
\end{equation}
We have the algebras of polynomials and the symmetric polynomials: 
\begin{equation}\label{EPolSi}
\Pol_{(\be^p)}=\O[y_{\be,1},\dots,y_{\be,p}]\quad\text{and}\quad
\Sym_{(\be^p)}=\Pol_{(\be^p)}^{\Si_p}
\end{equation}
While it is clear that the $y_{\be,r}$ commute, we do not yet know that they are algebraically independent, but this will turn out to be the case. For now, one can interpret $\Sym_{(\be^p)}$ as the algebra generated by the elementary symmetric functions in $y_{\be,1},\dots,y_{\be,p}$. Note using Hypothesis~\ref{HProp}(iv) that
\begin{equation}\label{EUpperBoundp}
\qdim \Sym_{(\be^p)}\leq \prod_{s=1}^p\frac{1}{1-q_\be^{2s}}.
\end{equation}

Given $\al \in Q_+$ of height $d$ and a root partition $\pi=(\be_1^{p_1}, \dots, \be_N^{p_N}) \in \Pi(\al)$ we define the parabolic subgroup 
\begin{align*}
\Si_\pi := \Si_{\height(\be_1)}^{\times p_1} \times \dots \times \Si_{\height(\be_N)}^{\times p_N} \subseteq \Si_d,
\\
\Si_{(\pi)} := \Si_{p_1\height(\be_1)} \times \dots \times \Si_{p_N\height(\be_N)} \subseteq \Si_d,
\end{align*}
and we denote by $\Si^\pi$ (resp. $\Si^{(\pi)}$) the set of minimal left coset representatives of $\Si_\pi$ (resp. $\Si_{(\pi)}$) in $\Si_d$. Set 
\begin{align*}
\BB_\pi &:= \{\psi_w \iota_\pi(b_1\otimes \dots\otimes b_N)\mid w\in \Si^\pi,\ b_n\in\BB_{\be_n^{\boxtimes p_n}}\ \text{for $n=1,\dots,N$}\}. 
\end{align*}
Using the natural embedding of $L({\be_1})^{\boxtimes p_1}\boxtimes\dots\boxtimes L({\be_N})^{\boxtimes p_N} \subseteq \bar\De(\pi)$, we define the elements
$$
v_\pi^-= (v_{\be_1}^-)^{\otimes p_1}\otimes\dots\otimes (v_{\be_N}^-)^{\otimes p_N} \in \bar\De(\pi)
$$
which belong to the word space corresponding to the words
$$
\bi_\pi:=\bi_{\be_1}^{p_1}\dots \bi_{\be_N}^{p_N}.
$$
From definitions we have

\begin{Lemma} \label{LDeBasis}
Let $\pi\in\Pi(\al)$. Then 
$\{bv_\pi^-\mid b\in \BB_\pi\}$
is a basis for $\bar \De(\pi)$. 
\end{Lemma}

Define 
\begin{align*}
\de_\pi &:= \iota_\pi(\de_{(\be_1^{p_1})}\otimes \dots \otimes \de_{(\be_N^{p_N})}),
\\
D_\pi &:= \iota_\pi(D_{(\be_1^{p_1})}\otimes  \dots \otimes D_{(\be_N^{p_N})}),
\\
e_\pi &:=   \iota_\pi(e_{(\be_1^{p_1})} \otimes \dots \otimes e_{(\be_N^{p_N})})
=D_\pi \de_\pi,\\
  \De(\pi) &:= ((R_\al e_\pi + I_{>\pi})/I_{>\pi})\<\deg(v_\pi^-)\>,\\
  \De'(\pi) &:= ((e_\pi R_\al + I_{>\pi})/I_{>\pi})\<\deg(v_\pi^+)\>,\\
  \Sym_\pi &:= \iota_\pi(\Sym_{(\be_1^{p_1})}\otimes\dots\otimes \Sym_{(\be_N^{p_N})}).
\end{align*}

Note by (\ref{EUpperBoundp}) and (\ref{ELPi}) that 
\begin{equation}\label{EDimUpper}
\qdim \Sym_\pi\leq l_\pi. 
\end{equation}
Choose also a homogeneous basis $X_\pi$ for $\Sym_\pi$.
The following lemma is a consequence of Hypothesis~\ref{HProp}(ii),(vi) and (\ref{ETauInv1}).
\begin{Lemma}\label{LDDelTau}
We have $D_\pi^\tau = D_\pi$ and $\de_\pi^\tau = \de_\pi$.
\end{Lemma}


\subsection{Powers of a single root}\label{SSPower}

Throughout this subsection $\be \in \Phi_+$ and $p \in \Z_{>0}$ are fixed. Define $\al := p\be$, and $\si:=(\be^p) \in \Pi(\al)$. 

Define $\bar R_{\al} := R_{\al} / I_{>\si}$, and given $r \in R_{\al}$ write $\bar r$ for its image in $\bar R_{\al}$. The following proposition is the main result of this subsection.
\begin{Proposition}\label{LPowerBasis} We have that
  \begin{enumerate}
    \item[\textrm{(i)}] $\{\bar b \bar f \bar e_{\si} \mid b \in \BB_\si, f \in X_\si\}$ is an $\O$-basis for $\De(\si)$.
    \item[\textrm{(ii)}] $\{\bar e_{\si} \bar f \bar D_{\si}\, \bar b^\tau \mid b \in \BB_\si, f \in X_\si\}$ is an $\O$-basis for $\De'(\si)$.
    \item[\textrm{(iii)}] $\{\bar b \bar e_{\si} \bar f \bar D_{\si} (\bar b')^\tau \mid b,b' \in \BB_\si, f \in X_\si\}$ is an $\O$-basis for $\bar R_{\al}$.
    \item[\textrm{(iv)}] The elements $\bar y_{\be,1},\dots,\bar y_{\be,p}$ are algebraically independent.
  \end{enumerate}
\end{Proposition}

The proof of the Proposition will occupy this subsection. It goes by induction on $p\,\height(\be)$. If $\be$ a simple root, then $R_\al=\bar R_\al$ is exactly the nil-Hecke algebra, and we are done by Theorem~\ref{TnilHeckeCellBasis}. For the rest of the section, we assume the Proposition holds with $\si = (\ga^s) \in \Pi(s\ga)$ whenever $\ga\in\Phi_+$ and $s\,\height(\ga) < p\,\height(\be)$ and prove that it also holds for $\si = (\be^p)$. We shall also assume that $\O=F$ is a field, and then use Lemma~\ref{PFieldToZ} to lift to $\Z$-forms.

\begin{Lemma}\label{L050413_1}
  Assume that $p=1$. Then Proposition~\ref{LPowerBasis} holds.
\end{Lemma}
\begin{proof}
  Since $L(\be)$ is the unique simple module in $\mod{\bar R_\be}$ and 
$$\Hom_{\bar R_\be}(\De(\be), L(\be)) = e_\be L(\be) = Fv_\be^-$$
is one-dimensional by Hypothesis~\ref{HProp}, it follows that $\De(\be)$ is the projective cover of $L(\be)$ in $\mod{\bar R_\be}$ under the map $\bar e_\be \mapsto v_\be^-$. All composition factors of $\De(\be)$ are isomorphic to $L(\be)$. Therefore, lifting the basis $\{b v_\be^- \mid b \in \BB_\be\}$ of $L(\be)$ to $\De(\be)$ we see that $\De(\be)$ is spanned by 
$$\{\bar b \phi(\bar e_\be) \mid b \in \BB_\be, \phi \in \End_{\bar R_\be}(\De(\be))\}.$$ 
By Hypothesis~\ref{HProp}(v), $\End_{\bar R_\be}(\De(\be)) \simeq \bar e_\be \bar R_\be \bar e_\be$ is generated by $\bar e_\be\bar y_\be \bar e_\be=\bar y_\be \bar e_\be$. Thus 
$$\De(\be) = F\text{-span}\{\bar b \bar f \bar e_\be \mid b \in \BB_\be, f \in X_\be \}.$$

  Analogously, $\De'(\be)$ is the projective cover of $L(\be)^\tau$ as right $\bar R_\be$-modules under the map $\bar e_\be \mapsto v_\be^+$. As above, lifting the basis $\{v_\be^+ D_\be b^\tau \mid b \in \BB_\be\}$ of $L(\be)^\tau$ to $\De'(\be)$ we see that 
  $$\De'(\be) = F\operatorname{-span}\{\bar e_\be \bar f \bar D_\be\, \bar b^\tau \mid b \in \BB_\be, f \in X_\be\}.$$ Therefore by Lemma~\ref{CProp} and Hypothesis~\ref{HProp}(iv), 
  $$\bar R_\be = \bar R_\be \bar e_\be \bar R_\be = F\text{-span}\{\bar b \bar e_\be \bar f \bar D_\be (\bar b')^\tau \mid b,b' \in \BB_\be, f \in X_\be \}.$$
  
  Let $\pi = (\be_1^{p_1},\dots,\be_N^{p_N}) > (\be)$. By definition and \cite[Proposition 2.16]{KL1}  we have
  \begin{align*}
      I_\pi 
      &= R_\be e(\bi_\pi) R_\be + I_{>\pi}\\
      &= \sum_{u,v \in \Si^{(\pi)}} \psi_u R_\pi e(\bi_\pi) R_\pi \psi_v^\tau + I_{>\pi}
        \subseteq \sum_{u,v \in \Si^{(\pi)}} \psi_u R_\pi \psi_v^\tau + I_{>\pi},
  \end{align*}
because $e(\bi_\pi) \in R_\pi$. The opposite inclusion follows from Lemma~\ref{LIotaImage}. 

For $n=1,\dots,N$, define $$B_n := \{b e_{(\be_n^{p_n})} f D_{\be_n^{p_n}} (b')^\tau \mid b,b' \in \BB_{(\be_n^{p_n})}, f \in X_{(\be_n^{p_n})}\}.$$ By part (iii) of the induction hypothesis, for $n=1,\dots,N$, the image of $B_n$ in $\bar R_{p_n \be_n}$ is a basis. Let
$$
B_\pi:=\{\iota_\pi(b_1\otimes\dots\otimes b_N)\mid b_1\in\BB_{(\be_1^{p_1})},\dots, b_N\in\BB_{(\be_N^{p_N})}\}.
$$
By Corollary~\ref{LParaIdeal2} and definitions from Section~\ref{SNota},
\begin{align*}
R_\pi + I_{>\pi} &= F\operatorname{-span}\{\iota_\pi(r_1\otimes \dots\otimes r_N) \mid r_n \in B_n \text{ for } n=1,\dots,N\} + I_{>\pi}\\
     &= F\operatorname{-span}\{b e_\pi f D_\pi (b')^\tau \mid b, b' \in B_\pi, f \in X_\pi\} + I_{>\pi}
\end{align*}
and therefore
$$
   I_\pi = F\operatorname{-span}\{\psi_u b e_\pi f D_\pi (b')^\tau \psi_v^\tau \mid u,v \in \Si^{(\pi)}, b, b' \in B_\pi, f \in X_\pi\} + I_{>\pi}.
$$
  By definition of $\BB_\pi$ we have
    \begin{align}
      I_\pi &= F\text{-span}\{b e_\pi f D_\pi (b')^\tau \mid b, b' \in \BB_\pi, f \in X_\pi\} + I_{>\pi},\label{EIpi}\\
      R_\be &= \sum_{\pi \in \Pi(\be)} F\text{-span}\{b e_\pi f D_\pi (b')^\tau \mid b, b' \in \BB_\pi, f \in X_\pi\}.
    \end{align}
Using (\ref{EDimUpper}) and the equality $\deg(D_\pi) = 2\deg(v_\pi^-)$ for all $\pi \in \Pi(\be)$, we get
\begin{align*}
  \dim_q(R_\be) &= \sum_{\pi\in\Pi(\be)}\dim_q(F\text{-span}\{b e_\pi f D_\pi (b')^\tau \mid b, b' \in \BB_\pi, f \in X_\pi\})\\
  &\leq \sum_{\pi\in\Pi(\be)}\Big(\sum_{b \in \BB_\pi}q^{\deg(b)}\Big) \dim_q(\Sym_\pi) q^{\deg(D_\pi)} \Big(\sum_{b \in \BB_\pi}q^{\deg(b)}\Big)\\
  &\leq \sum_{\pi\in\Pi(\be)}\Big(\sum_{b \in \BB_\pi}q^{\deg(b v_\pi^-)}\Big)^2 l_\pi\\
  &= \sum_{\pi\in\Pi(\be)}\dim_q(\bar \De(\pi))^2 l_\pi = \dim_q(R_\be),
\end{align*}
by Corollary~\ref{TDim}. The inequalities are therefore equalities, and this implies that the spanning set $\{b e_\pi f D_\pi (b')^\tau \mid \pi \in \Pi(\be), b, b' \in \BB_\pi, f \in X_\pi\}$ of $R_\be$ is a basis and $\qdim \Sym_{\pi}=l_\pi$ for all $\pi$. These yield (iii) and (iv) of Proposition~\ref{LPowerBasis} in our special case $p=1$.

To show (i) and (ii), we have already noted that the claimed bases span $\De(\be)$ and $\De'(\be)$, respectively. We now apply part (iii) to see that they are linearly independent.
\end{proof}

\begin{Corollary}\label{CFree}
We have
  \begin{enumerate}
    \item[\textrm{(i)}]  $\bar e_\be \bar R_\be \bar e_\be$ is a polynomial algebra in the variable $\bar y_\be \bar e_\be$.
    \item[\textrm{(ii)}]  $\De(\be)$ is a free right $\bar e_\be \bar R_\be \bar e_\be$-module with basis $\{\bar b \bar e_\be \mid b \in \BB_\be\}$.
    \item[\textrm{(iii)}]  $\De'(\be)$ is a free left $\bar e_\be \bar R_\be \bar e_\be$-module with basis $\{\bar e_\be \bar D_\be \bar b^\tau \mid b \in \BB_\be\}$.
  \end{enumerate}
\end{Corollary}
\begin{proof}
By the lemma, we have Proposition~\ref{LPowerBasis} for $p=1$. Now, (i) follows  from parts (i) and (iv) of the proposition. 
The remaining statements follow from parts (i) and (ii) of the proposition.
\end{proof}

\begin{Corollary}\label{CGrot}
  In the Grothendieck group, we have $[\De(\be)] = [L(\be)]/(1-q_\be^2).$
\end{Corollary}

\begin{Lemma}\label{LInduct} Up to a degree shift, 
  $\De(\be)^{\circ p} \cong \bar R_{p\be} \bar e_{\be^{\boxtimes p}}$.
\end{Lemma}
\begin{proof}
By Corollary~\ref{COneRootIdeal} we have a map
$$
\De(\be)^{\boxtimes p} \to \Res_{\be, \dots, \be}(\bar R_{p\be} \bar e_{\be^{\boxtimes p}}),\; \bar e_\be^{\otimes p} \mapsto \bar e_{\be^{\boxtimes p}}.
$$
By Frobenius reciprocity, we obtain a map
$$
\mu: \De(\be)^{\circ p} \to \bar R_{p\be} \bar e_{\be^{\boxtimes p}},\; 1_{\be, \dots, \be} \otimes \bar e_\be^{\otimes p} \mapsto \bar e_{\be^{\boxtimes p}}.
$$
We now show that $I_{>(\be^p)} \De(\be)^{\circ p} = 0$. It is enough to prove that $\Res_\pi \De(\be)^{\circ p} = 0$ for all $\pi > (\be^p)$. Since all composition factors of $\De(\be)$ are isomorphic to $L(\be)$, it follows that all composition factors of $\De(\be)^{\circ p}$ are isomorphic to $L(\be)^{\circ p} \cong L(\be^p)$. By Theorem~\ref{TStand}\textrm{(vi)}, $\Res_\pi(L(\be)^{\circ p}) = 0$, which proves the claim. Since $e_{\be^{\boxtimes p}} 1_{\be, \dots, \be} \otimes \bar e_\be^{\otimes p} = 1_{\be, \dots, \be} \otimes \bar e_\be^{\otimes p}$, we obtain a map 
\begin{align*}
\nu: \bar R_{p\be} \bar e_{\be^{\boxtimes p}} &\to \De(\be)^{\circ p},\; \bar e_{\be^{\boxtimes p}} \mapsto 1_{\be, \dots, \be} \otimes \bar e_\be^{\otimes p}.
\end{align*}
 The homomorphisms $\mu, \nu$ map the evident cyclic generators to each other, and so are inverse isomorphisms.
\end{proof}

\begin{Lemma}\label{LMackeyPsi}
  There exists an endomorphism of $\De(\be) \circ \De(\be)$ which sends $1_{\be,\be} \otimes (\bar e_\be \otimes \bar e_\be)$ to $ \psi_{\be,1} 1_{\be,\be} \otimes (\bar e_\be \otimes \bar e_\be)$.
\end{Lemma}
\begin{proof}
Apply the Mackey theorem to $\Res_{\be,\be}(\De(\be) \circ \De(\be))$. We get a short exact sequence of $R_\be \boxtimes R_\be$-modules
$$
0 \to \De(\be) \boxtimes \De(\be) \to \Res_{\be,\be}(\De(\be) \circ \De(\be)) \to (\De(\be) \boxtimes \De(\be))\<-\be\cdot\be\> \to 0,
$$
where $\psi_{\be,1} 1_{\be,\be} \otimes (\bar e_\be \otimes\bar e_\be)\in \Res_{\be,\be}(\De(\be) \circ \De(\be))$ is a preimage of the standard generator of $(\De(\be) \boxtimes \De(\be))\<-\be\cdot\be\>$.

We now show that this is actually a sequence of $\bar R_\be \boxtimes \bar R_\be$-modules. It is sufficient to show that for any $\pi > (\be)$, we have that
$$
\Res_{\pi,\be} \circ \Res_{\be,\be}(\De(\be) \circ \De(\be)) = 0 = \Res_{\be,\pi} \circ \Res_{\be,\be}(\De(\be) \circ \De(\be)).
$$
We show the first equality, the second being similar. All composition factors of $\De(\be)$ are isomorphic to $L(\be)$, so all composition factors of $\De(\be) \circ \De(\be)$ are isomorphic to $L(\be) \circ L(\be)$, and thus all composition factors of $\Res_{\be,\be}(\De(\be) \circ \De(\be))$ are isomorphic to $L(\be) \boxtimes L(\be)$. Theorem~\ref{TStand} now tells us that $\Res_{\pi}(L(\be)) = 0$ for all $\pi > (\be)$.

By the projectivity of $\De(\be)$ as $\bar R_\be$-module, the short exact sequence splits, giving the required endomorphism by Frobenius reciprocity.
\end{proof}

\begin{Corollary}\label{C030413_1}
  $\bar \psi_{\be,1} \bar e_{\be^{\boxtimes 2}} = \bar e_{\be^{\boxtimes 2}} \bar \psi_{\be,1} \bar e_{\be^{\boxtimes 2}}$.
\end{Corollary}
\begin{proof}
Let $\phi$ be the endomorphism of $\De(\be) \circ \De(\be)$ constructed in Lemma~\ref{LMackeyPsi}, regarded as an endomorphism of $\bar R_{2\be} \bar e_{\be^{\boxtimes 2}}$ by Lemma~\ref{LInduct}. Then
$$
  \bar \psi_{\be,1} \bar e_{\be^{\boxtimes 2}} = \phi(\bar e_{\be^{\boxtimes 2}}) = \phi(\bar e_{\be^{\boxtimes 2}}^2) = \bar e_{\be^{\boxtimes 2}} \phi(\bar e_{\be^{\boxtimes 2}}) = \bar e_{\be^{\boxtimes 2}} \bar \psi_{\be,1} \bar e_{\be^{\boxtimes 2}},$$
as required. 
\end{proof}

\begin{Corollary} \label{C030413_2} 
We have $\bar e_{\be^{\boxtimes p}}\bar e_{(\be^p)} \bar e_{\be^{\boxtimes p}}=\bar e_{(\be^p)}$. 
\end{Corollary}
\begin{proof}
Follows from (\ref{EDeSi}), Corollary~\ref{C030413_1} and Hypothesis~\ref{HProp}. 
\end{proof}

\begin{Lemma}\label{LEndBasis}
  The set $\{\bar e_{\be^{\boxtimes p}} \bar y_{\be,1}^{a_1} \dots \bar y_{\be,p}^{a_p} \bar \psi_{\be,w} \bar e_{\be^{\boxtimes p}} \mid w \in \Si_p, a_1, \dots, a_p \geq 0 \}$ gives a linear basis of $\bar e_{\be^{\boxtimes p}} \bar R_{\al} \bar e_{\be^{\boxtimes p}}$. 
\end{Lemma}
\begin{proof}
The elements above are linearly independent by Lemmas~\ref{LInduct} and \ref{LMackeyPsi}, and Corollary~\ref{CFree}. We use Frobenius reciprocity, Corollary~\ref{CGrot}, and \cite[Lemma 2.11]{BKM} to see that
\begin{align*}
  \qdim \End_{R_\al}(\De(\be)^{\circ p}) &= \qdim \Hom_{R_{\be,\dots,\be}}(\De(\be)^{\boxtimes p}, \Res_{\be,\dots,\be} \De(\be)^{\circ p})\\
  &\leq [\Res_{\be,\dots,\be} \De(\be)^{\circ p} : L(\be)^{\boxtimes p}]\\
  &= [\Res_{\be,\dots,\be} L(\be)^{\circ p} : L(\be)^{\boxtimes p}]/(1-q_\be^2)^p\\
  &= q_\be^{-\frac{1}{2}p(p-1)} [p]_\be^! / (1-q_\be^2)^p.
\end{align*}
By the formula for the Poincar\'e polynomial of $\Si_p$, we have shown that
\[
  \qdim \bar e_{\be^{\boxtimes p}} \bar R_{\al} \bar e_{\be^{\boxtimes p}} \leq \frac{\sum_{w \in \Si_p}q_\be^{-2l(w)}}{(1-q_\be^2)^p},
\]
showing that the proposed basis also spans.
\end{proof}

The next two lemmas are proved using ideas that already appeared in the proofs of \cite[Lemmas]{BKM}.

\begin{Lemma}\label{LNHRel1}
  We have that
\begin{align*}
  \bar \psi_{\be, r}^2 \bar e_{\be^{\boxtimes p}} &= 0, &\text{for } 1 \leq r \leq p-1,\\
  \bar \psi_{\be, r} \bar \psi_{\be, s} \bar e_{\be^{\boxtimes p}} &= \bar \psi_{\be, s} \bar \psi_{\be, r} \bar e_{\be^{\boxtimes p}}, &\text{for $|r-s| > 1$, and}\\
  \bar \psi_{\be, r} \bar \psi_{\be, r+1} \bar \psi_{\be, r} \bar e_{\be^{\boxtimes p}} &= \bar \psi_{\be, r+1} \bar \psi_{\be, r} \bar \psi_{\be, r+1} \bar e_{\be^{\boxtimes p}}, &\text{for } 1 \leq r \leq p-2.
\end{align*}
\end{Lemma}
\begin{proof}
  We use Lemma~\ref{LInduct} to identify $\bar R_{p\be} \bar e_{\be^{\boxtimes p}}$ with $\De(\be)^{\circ p}$. It is enough to prove the first relation in the case $p=2$. The Mackey theorem analysis in the proof of Lemma~\ref{LMackeyPsi} shows that, as a graded vector space
\begin{equation}\label{EDeSquare}
(\De(\be) \circ \De(\be))_{\bi_\be^2} = e(\bi_\be^2) \otimes (\De(\be) \boxtimes \De(\be)) \oplus \psi_{\be,1} e(\bi_\be^2) \otimes (\De(\be) \boxtimes \De(\be)).
\end{equation}
The vector $\bar e_\be \in \De(\be)_{\bi_\be}$ is of minimal degree, and thus $\psi_{\be,1} e(\bi_\be^2) \otimes (\bar e_\be \otimes \bar e_\be)$ is of minimal degree in $(\De(\be) \circ \De(\be))_{\bi_\be^2}$. The degree of $\psi_{\be,1}^2 e(\bi_\be^2) \otimes (\bar e_\be \otimes \bar e_\be)$ is smaller by $\be\cdot\be$, so the vector is zero.

  The second relation is clear from the definitions. To prove the third relation, it is sufficient to consider $p=3$. Let $w_r := w_{\be,r}$, and set $w_0:=w_1w_2w_1$. Using the defining relations of $R_{3\be}$, we deduce that $(\psi_{\be, 2} \psi_{\be, 1} \psi_{\be, 2} - \psi_{\be, 1} \psi_{\be, 2} \psi_{\be, 1}) e(\bi_\be^3) \otimes (\bar e_\be \otimes \bar e_\be \otimes \bar e_\be)$ is an element of degree $3\deg(v_\be^-)-6\be\cdot\be$ in $S:=\sum_{w < w_0} \psi_w e(\bi_\be^3) \otimes (\De(\be) \boxtimes \De(\be) \boxtimes \De(\be))$, where $<$ denotes the Bruhat order. By a Mackey theorem analysis as in the proof of Lemma~\ref{LMackeyPsi}, we see that
$$S = \sum_{w \in \{1,w_1,w_2,w_1w_2,w_2w_1\}} \psi_w e(\bi_\be^3) \otimes (\De(\be) \boxtimes \De(\be) \boxtimes \De(\be)).$$
The lowest degree of an element in $S$ is therefore $3\deg(v_\be^-)-4\be\cdot\be$, and the third relation is proved.
\end{proof}

\begin{Lemma}\label{LNHRel2}
  There exists a unique choice of $\eps_\be = \pm 1$ such that
\begin{align*}
  \bar \psi_{\be, r} \bar y_{\be, s} \bar e_{\be^{\boxtimes p}} &= \bar y_{\be, s} \bar \psi_{\be, r} \bar e_{\be^{\boxtimes p}}, &\text{for } s \neq r, r+1,\\
  \bar \psi_{\be, r} \eps_\be \bar y_{\be, r+1} \bar e_{\be^{\boxtimes p}} &= (\eps_\be \bar y_{\be, r} \bar \psi_{\be, r} + 1) \bar e_{\be^{\boxtimes p}}, &\text{for $1 \leq r < p$, and}\\
  \eps_\be \bar y_{\be, r+1} \bar \psi_{\be, r} \bar e_{\be^{\boxtimes p}} &= (\bar \psi_{\be, r} \eps_\be \bar y_{\be, r} + 1) \bar e_{\be^{\boxtimes p}}, &\text{for $1 \leq r < p$.}
\end{align*}
\end{Lemma}
\begin{proof}
  The first relation is clear from the definitions. It is enough to prove the remaining relations for $p=2$. Using the defining relations of $R_{2\be}$ and a Mackey theorem analysis as in the proof of Lemma~\ref{LMackeyPsi}, we deduce that
\begin{align*}
  (\bar \psi_{\be, 1} \bar y_{\be, 2} - \bar y_{\be, 1} \bar \psi_{\be, 1})\bar e_{\be^{\boxtimes 2}} &\in \sum_{w<w_{\be,1}}\psi_we(\bi_\be^2) \otimes (\De(\be) \boxtimes \De(\be))\\
      &= e(\bi_\be^2) \otimes (\De(\be) \boxtimes \De(\be)),
\end{align*}
and the only vector of the correct degree is $\bar e_{\be^{\boxtimes 2}}$. Therefore (working over $\Z$) we must have that
$$(\bar \psi_{\be, 1} \bar y_{\be, 2} - \bar y_{\be, 1} \bar \psi_{\be, 1}) \bar e_{\be^{\boxtimes 2}} = c_+ \bar e_{\be^{\boxtimes 2}}$$
for some $c_+ \in \Z$. Similarly, we obtain
$$(\bar \psi_{\be, 1} \bar y_{\be, 1} - \bar y_{\be, 2} \bar \psi_{\be, 1}) \bar e_{\be^{\boxtimes 2}} = c_- \bar e_{\be^{\boxtimes 2}}$$
for some $c_- \in \Z$. We compute
\begin{align*}
(\bar \psi_{\be,1} \bar y_{\be,1} \bar y_{\be,2} - \bar y_{\be,1} \bar y_{\be,2} \bar \psi_{\be,1}) \bar e_{\be^{\boxtimes 2}} &= (\bar y_{\be,2} \bar \psi_{\be,1} + c_-)\bar y_{\be,2} - \bar y_{\be,2} (\bar \psi_{\be,1} \bar y_{\be,2} - c_+) \\
    &= (c_-+c_+)\bar y_{\be,2} \bar e_{\be^{\boxtimes 2}}\\
(\bar \psi_{\be,1} \bar y_{\be,1} \bar y_{\be,2} - \bar y_{\be,1} \bar y_{\be,2} \bar \psi_{\be,1}) \bar e_{\be^{\boxtimes 2}} &= (\bar y_{\be,1} \bar \psi_{\be,1} + c_+)\bar y_{\be,1} - \bar y_{\be,1} (\bar \psi_{\be,1} \bar y_{\be,1} - c_-) \\
    &= (c_++c_-)\bar y_{\be,1} \bar e_{\be^{\boxtimes 2}}
\end{align*}
and since $\bar y_{\be,1} \bar e_{\be^{\boxtimes 2}}$ and $\bar y_{\be,1} \bar e_{\be^{\boxtimes 2}}$ are linearly independent by Lemma~\ref{LEndBasis}, we must have $c_-=-c_+$. We now fix a prime $p$ and extend scalars to $\FF_p$. Suppose that $\eps_\be = 0 \in \FF_p$, so that
\begin{align*}
  \bar \psi_{\be, 1} \bar y_{\be, 2} \bar e_{\be^{\boxtimes 2}} &= \bar y_{\be, 1} \bar \psi_{\be, 1} \bar e_{\be^{\boxtimes 2}}\\
  \bar \psi_{\be, 1} \bar y_{\be, 1} \bar e_{\be^{\boxtimes 2}} &= \bar y_{\be, 2} \bar \psi_{\be, 1} \bar e_{\be^{\boxtimes 2}}.
\end{align*}
 Define $S$ to be the submodule of $\De(\be) \circ \De(\be)$ generated by $\bar y_{\be,1} \bar e_{\be^{\boxtimes 2}}$ and $\bar y_{\be,2} \bar e_{\be^{\boxtimes 2}}$. The above equations show that the endomorphism defined by right multiplication by $\bar \psi_{\be,1} \bar e_{\be^{\boxtimes 2}}$ leaves $S$ invariant. On the other hand, $\De(\be) \circ \De(\be) / S \cong L(\be) \circ L(\be)$ is irreducible. Since the endomorphism algebra of an irreducible module is one dimensional, we have a contradiction. Therefore $\eps_\be \neq 0$ when reduced modulo any prime, i.e. $\eps_\be = \pm 1$.
\end{proof}

\begin{Corollary}\label{C030413_3}
  The homomorphism from the nilHecke algebra $H_p$ determined by
$$
\zeta: H_p \to \bar e_{\be^{\boxtimes p}} \bar R_{\al} \bar e_{\be^{\boxtimes p}}, y_r \mapsto \eps_\be \bar y_{\be, r} \bar e_{\be^{\boxtimes p}}, \psi_r \mapsto \bar \psi_{\be, r} \bar e_{\be^{\boxtimes p}}
$$
 is an isomorphism. Under this isomorphism the idempotent $e_p\in H_p$ is mapped onto $\bar e_\si$. 
\end{Corollary}
\begin{proof}
  Using Lemmas~\ref{LNHRel1} and~\ref{LNHRel2}, we see that the map exists. By Lemma~\ref{LEndBasis}, the map is an isomorphism. The second statement now follows using Corollary~\ref{C030413_2}. 
\end{proof}

\begin{Corollary}\label{CSymm}
Given $f \in \Sym_\si$, $\bar f$ commutes with $\bar \de_\si$, $\bar e_\si$, and $\bar e_\si \bar D_\si$.
\end{Corollary}
\begin{proof}
  It follows directly from Hypothesis~\ref{HProp}(iv) and the definitions that $\de_\si$ commutes with every element of $\Pol_\si$, and in particular with every element of the subalgebra $\Sym_\si$.
Denote by $w_0$ the longest element of $\Si_p$. Then by Corollaries~\ref{C030413_1} and~\ref{C030413_3}
\begin{align*}
  \bar e_\si \bar D_\si &= \bar \psi_{\be,w_0} \bar y_{\be,2} \dots \bar y_{\be,p}^{p-1}  \bar e_{\be^{\boxtimes p}} \bar \psi_{\be,w_0} \iota(D_\be \otimes \dots \otimes D_\be)\\
    &= (\bar e_{\be^{\boxtimes p}} \psi_{\be,w_0} \bar e_{\be^{\boxtimes p}})(\bar y_{\be,2} \dots \bar y_{\be,p}^{p-1})(\bar e_{\be^{\boxtimes p}} \psi_{\be,w_0} \bar e_{\be^{\boxtimes p}})\iota(D_\be \otimes \dots \otimes D_\be)\\
    &= \zeta(\psi_{w_0})(\bar y_{\be,2} \dots \bar y_{\be,p}^{p-1})\zeta(\psi_{w_0})\iota(D_\be \otimes \dots \otimes D_\be)
\end{align*}
  Any $f \in \Sym_\si$ commutes with $\iota(D_\be \otimes \dots \otimes D_\be)$ by Hypothesis~\ref{HProp}(iv). It is well known that the center of the nilHecke algebra $H_p$ is given by the symmetric functions $\Sym_p$. In particular, every element of $\Sym_p$ commutes with $\psi_{w_0}$. Let $g \in \Sym_p$ be such that $\zeta(g) = \bar f \bar e_{\be^{\boxtimes p}}$. Then $\zeta(\psi_{w_0}) \bar f = \zeta(\psi_{w_0} g) = \zeta(g \psi_{w_0}) = f \zeta(\psi_{w_0})$. This implies the claim.
\end{proof}

We can now finish the proof of Proposition~\ref{LPowerBasis}. Corollary~\ref{C030413_3} provides an isomorphism $H_p \cong \End_{R_\al}(\bar R_\al \bar e_{\be^{\boxtimes p}})$ under which the idempotent $e_p$ corresponds to right multiplication by $\bar e_\si$. But $e_p$ is a primitive idempotent, so the image $\bar R_\al \bar e_\si = \De(\si)$ of this endomorphism is an indecomposable projective $\bar R_\al$-module. We may identify
\begin{equation}\label{EEndo}
\End_{R_\al}(\De(\si)) \cong \bar e_\si \bar R_\al \bar e_\si = \zeta(e_p H_p e_p) = \zeta(\Sym_p e_p) = \bar e_\si \bar \Sym_\si \bar e_\si \cong \Sym_\si,
\end{equation}
where the action of $\Sym_\si$ on
$\De(\si)=\bar R_\al \bar e_\si$
is given by right multiplication which makes sense in view of  Corollary~\ref{CSymm}. Therefore $\De(\si) \onto L(\si), \bar e_\si \mapsto v_\si^-$ is a projective cover in $\mod{\bar R_{\al}}$. Furthermore, since $\mod{\bar R_{\al}}$ has only one irreducible module, every composition factor of $\De(\si)$ is isomorphic to $L(\si)$ with an appropriate degree shift. We can lift the basis $\{b v_\si^- \mid b \in \BB_\si\}$ for $L(\si)$ to the set $\{\bar b \bar e_\si \mid b \in \BB_\si\} \subseteq \De(\si)$. Using the basis $X_\si$ for $\Sym_\si$, we get a basis $\{\bar b \bar f \bar e_\si \mid b \in \BB_\si, f \in X_\si\}$ for $\De(\si)$.

Similarly, $\De'(\si) \onto L(\si)^\tau, \bar e_\si \mapsto v_\si^+$ is a projective cover in $\mod{\bar R_{\al}^{op}}$. It is immediate that $\{v_\si^+ D_\si b^\tau \mid b \in \BB_\si\}$ is a basis of $L(\si)^\tau$. Lifting as above, we have that $\{\bar e_\si \bar f\bar D_\si \bar b^\tau \mid b \in \BB_\si, f \in X_\si\}$ is a basis for $\De'(\si)$.

Finally, applying the multiplication map and Corollary~\ref{CSymm} we have that $\{\bar b \bar e_\si \bar f \bar D_\si \bar (b')^\tau \mid b,b' \in \BB_\si, f \in X_\si\}$ spans $\bar R_{\al}$. Therefore by induction, 
$$R_{\al} = F\operatorname{-span}\{\bar b \bar e_\pi \bar f \bar D_\pi \bar (b')^\tau \mid \pi \in \Pi(\al), b,b' \in \BB_\pi, f \in X_\pi\},$$
and comparing graded dimensions with Corollary~\ref{TDim} as in the proof of Lemma~\ref{L050413_1}, this set is therefore a basis.
\qed

\subsection{General case}

In this section we use the results of the previous subsections to obtain affine cellular bases of the KLR algebras of finite type. Fix $\al \in Q_+$ and $\pi=(\be_1^{p_1}, \dots, \be_N^{p_N}) \in \Pi(\al)$. Define $\bar R_\al := R_\al / I_{>\pi}$, and write $\bar r \in \bar R_\al$ for the image of an element $r \in R_\al$.

We begin with some easy consequences of the previous section.
\begin{Corollary}\label{CSymmGen}
We have
 \begin{enumerate}
  \item[\textrm{(i)}] Given $f \in \Sym_\pi$, $\bar f$ commutes with $\bar \de_\pi$, $\bar e_\pi$, and $\bar e_\pi \bar D_\pi$.
  \item[\textrm{(ii)}] Up to a grading shift, $\De(\pi) \cong \De(\be_1^{p_1}) \circ \dots \circ \De(\be_N^{p_N})$.
  \item[\textrm{(iii)}] The map $\Sym_\pi \to \End_{\bar R_\al}(\De(\pi))$ sending $f$ to right multiplication by $\bar e_\pi \bar f \bar e_\pi$ is an isomorphism of algebras.
  \item[\textrm{(iv)}] The map $\Sym_\pi \to \End_{\bar R_\al}(\De'(\pi))$ sending $f$ to left multiplication by $\bar e_\pi \bar f \bar e_\pi$ is an isomorphism of algebras.
 \end{enumerate}
\end{Corollary}
\begin{proof}
Claim (i) follows directly from Corollary~\ref{CSymm} and the definitions.

The proof of claim (ii) is similar to that of Lemma~\ref{LInduct}. To be precise, by Corollary~\ref{LParaIdeal2} we have a map
$$
\De(\be_1^{p_1}) \boxtimes \dots \boxtimes \De(\be_N^{p_N}) \to \Res_\pi\De(\pi),\; \bar e_{(\be_1^{p_1})}\otimes \dots \otimes \bar e_{(\be_N^{p_N})} \mapsto \bar e_\pi,
$$
which by Frobenius reciprocity determines a homomorphism
$$
\mu: \De(\be_1^{p_1}) \circ \dots \circ \De(\be_N^{p_N}) \to \De(\pi),\; 1_\pi \otimes (\bar e_{(\be_1^{p_1})}\otimes \dots \otimes \bar e_{(\be_N^{p_N})}) \mapsto \bar e_\pi.
$$

We now claim that $I_{>\pi} (\De(\be_1^{p_1}) \circ \dots \circ \De(\be_N^{p_N})) = 0$. It is enough to prove that $\Res_\si(\De(\be_1^{p_1}) \circ \dots \circ \De(\be_N^{p_N})) = 0$ for all $\si > \pi$. By exactness of induction, it follows that $\De(\be_1^{p_1}) \circ \dots \circ \De(\be_N^{p_N})$ has an exhaustive filtration by $L(\be_1^{p_1}) \circ \dots \circ L(\be_N^{p_N}) = \bar \De(\pi)$. By Theorem~\ref{TStand}\textrm{(vi)}, $\Res_\si(\bar \De(\pi)) = 0$, which proves the claim.  

Since 
$$e_\pi 1_\pi \otimes (\bar e_{(\be_1^{p_1})}\otimes \dots \otimes \bar e_{(\be_N^{p_N})}) = 1_\pi \otimes (\bar e_{(\be_1^{p_1})}\otimes \dots \otimes \bar e_{(\be_N^{p_N})}),$$
we obtain a map 
\begin{align*}
\nu: \De(\pi) &\to \De(\be_1^{p_1}) \circ \dots \circ \De(\be_N^{p_N}),\; \bar e_\pi \mapsto 1_\pi \otimes (\bar e_{(\be_1^{p_1})}\otimes \dots \otimes \bar e_{(\be_N^{p_N})}).
\end{align*}
 The homomorphisms $\mu, \nu$ map the evident cyclic generators to each other, and so are inverse isomorphisms.

We use claim (ii) to identify $\De(\pi)$ with $\De(\be_1^{p_1}) \circ \dots \De(\be_N^{p_N})$. As noted in the proof of claim (ii), $\De(\pi)$ has an exhaustive filtration by 
$$\bar \De(\pi) = \oplus_{w \in \Si^{(\pi)}} \psi_w 1_\pi \otimes (\bar \De(\be_1^{p_1}) \boxtimes \dots \boxtimes \bar \De(\be_N^{p_N})).$$
By Theorem~\ref{TStand}(vi), $\Res_\pi\bar \De(\pi)$ picks out the summand corresponding to $w=1$. Therefore $\Res_\pi\De(\pi) \cong \De(\be_1^{p_1}) \boxtimes \dots \boxtimes \De(\be_N^{p_N})$. Applying Frobenius reciprocity and (\ref{EEndo}), we obtain
\begin{align*}
  \End_{R_\al}(\De(\pi)) &= \Hom_{R_\al}(\De(\be_1^{p_1}) \circ \dots \circ \De(\be_n^{p_N}), \De(\be_1^{p_1}) \circ \dots \circ \De(\be_n^{p_N}))\\
     &\simeq \Hom_{R_\pi}(\De(\be_1^{p_1}) \boxtimes \dots \boxtimes \De(\be_n^{p_N}), \Res_\pi(\De(\be_1^{p_1}) \circ \dots \circ \De(\be_n^{p_N})))\\
     &\simeq \End_{R_\pi}(\De(\be_1^{p_1}) \boxtimes \dots \boxtimes \De(\be_N^{p_N}))\\
     &\simeq \End_{R_{p_1\be_1}}(\De(\be_1^{p_1})) \otimes \dots \otimes \End_{R_{p_1\be_N}}(\De(\be_1^{p_N}))\\
     &\simeq \Sym_{(\be_1^{p_1})} \otimes \dots \otimes \Sym_{(\be_N^{p_N})} \simeq \Sym_\pi.
\end{align*}
This proves claim (iii), and claim (iv) is shown similarly.
\end{proof}

\begin{Proposition}\label{LGenBasis}
  We have that
  \begin{enumerate}
    \item[\textrm{(i)}] $\{\bar b \bar f \bar e_\pi \mid b \in \BB_\pi, f \in X_\pi\}$ is an $\O$-basis for $\De(\pi)$,
    \item[\textrm{(ii)}] $\{\bar e_{\pi} \bar f \bar D_{\pi} \bar b^\tau \mid b \in \BB_\pi, f \in X_\pi\}$ is an $\O$-basis for $\De'(\pi)$, and
    \item[\textrm{(iii)}] $\{\bar b \bar e_{\pi} \bar f \bar D_{\pi} (\bar b')^\tau \mid b,b' \in \BB_\pi, f \in X_\pi\}$ is an $\O$-basis for $\bar I_\pi$.
  \end{enumerate}
\end{Proposition}
\begin{proof}
For $n=1,\dots,N$, define
$$B_n := \{\bar b \bar f \bar e_{(\be_n^{p_n})} \mid b \in \BB_{(\be_n^{p_n})}, f \in X_{(\be_n^{p_n})}\}.$$
By Proposition~\ref{LPowerBasis}, $B_n$ is a basis of $\De(\be_n^{p_n})$ for each $n=1,\dots,N$. 
Let $\bar \iota_\pi: \bar R_{p_1 \be_1} \otimes \dots \otimes \bar R_{p_N \be_N} \to \bar R_\al$ be the map induced by $\iota_\pi$, as in Corollary~\ref{LParaIdeal2}. 
Using \cite[Proposition 2.16]{KL1}, and computing as in the proof of Lemma~\ref{L050413_1}, we have
\begin{align*}
  \De(\pi) &= \sum_{w \in \Si^{(\pi)}} \bar \psi_w \bar R_\pi \bar e_\pi = \sum_{w \in \Si^{(\pi)}} \bar \psi_w \bar \iota_\pi(\De(\be_1^{p_1}) \otimes \dots \otimes \De(\be_N^{p_N}))\\
    &= \O\operatorname{-span}\{\bar \psi_w \bar \iota_\pi(b_1 \otimes \dots \otimes b_N) \mid w \in \Si^{(\pi)}, b_n \in B_n\}\\
    &= \O\operatorname{-span}\{\bar b \bar f \bar e_\pi \mid b \in \BB_\pi, f \in X_\pi\}.
\end{align*}
We have shown that the set in (i) spans $\De(\pi)$. A similar argument shows that the set in (ii) spans $\De'(\pi)$. Now, applying the multiplication map $\De(\pi) \otimes \De'(\pi) \onto \bar I_\pi$ and using Corollary~\ref{CSymmGen}(i) yields the spanning set of (iii). Letting $\pi$ vary over $\Pi(\al)$, we have
$$
  R_\al = \sum_{\pi \in \Pi(\al)} \O\operatorname{-span}\{b e_\pi f D_\pi (b')^\tau \mid b, b' \in \BB_\pi, f \in X_\pi\}.
$$
Using (\ref{EDimUpper}) and the equality $\deg(D_\pi) = 2\deg(v_\pi^-)$ for all $\pi \in \Pi(\al)$, we get
\begin{align*}
  \dim_q(R_\al) &= \sum_{\pi\in\Pi(\al)}\dim_q(\O\operatorname{-span}\{b e_\pi f D_\pi (b')^\tau \mid b, b' \in \BB_\pi, f \in X_\pi\})\\
  &\leq \sum_{\pi\in\Pi(\al)}\Big(\sum_{b \in \BB_\pi}q^{\deg(b)}\Big) \dim_q(\Sym_\pi) q^{\deg(D_\pi)} \Big(\sum_{b \in \BB_\pi}q^{\deg(b)}\Big)\\
  &\leq \sum_{\pi\in\Pi(\al)}\Big(\sum_{b \in \BB_\pi}q^{\deg(b v_\pi^-)}\Big)^2 l_\pi\\
  &= \sum_{\pi\in\Pi(\al)}\dim_q(\bar \De(\pi))^2 l_\pi = \dim_q(R_\al),
\end{align*}
by Corollary~\ref{TDim}. The inequalities are therefore equalities, and this implies that the spanning set $\{b e_\pi f D_\pi (b')^\tau \mid \pi \in \Pi(\al), b, b' \in \BB_\pi, f \in X_\pi\}$ of $R_\al$ is a basis and $\qdim \Sym_{\pi}=l_\pi$ for all $\pi$.

To show (i) and (ii), we have already noted that the claimed bases span $\De(\be)$ and $\De'(\be)$, respectively. We now apply part (iii) to see that they are linearly independent.
\end{proof}

\begin{Corollary} \label{CCellBasis}
The set $\{ b e_{\pi} f  D_{\pi} ( b')^\tau \mid \pi\in \Pi(\al),\ b,b' \in \BB_\pi, f \in X_\pi\}$ is an $\O$-basis for $R_\al$.
\end{Corollary}
\begin{proof}
Apply Proposition~\ref{LGenBasis}(iii) and the fact that the filtration by the ideals $I_\pi$ exhausts $R_\al$, which follows from Lemma~\ref{LAllIdem}. 
\end{proof}

\section{Affine cellularity}
Recall the notion of an affine cellular algebra from the introduction. 
In this section, we fix $\al \in Q_+$ and prove that $R_\al$ is affine cellular over $\Z$ (which then implies that it is affine cellular over any $k$). 

For any $\pi \in \Pi(\al)$, we define 
$$I_\pi' := \Z\text{-span}\{b e_\pi \Sym_\pi D_\pi (b')^\tau \mid b, b' \in \BB_\pi\}.$$
By Corollary~\ref{CCellBasis}, we have $R_\al=\oplus_{\pi\in\Pi(\al)}I_\pi'$.  Moreover, $\tau(I_\pi') = I_\pi'$. Indeed, $\de_\pi$ commutes with elements of $\Sym_\pi$ in view of Hypothesis~\ref{HProp}(iv). So by Lemma~\ref{LDDelTau}, we have
  \begin{align*}
    \tau(I_\pi') &= \Z\text{-span}\{b' D_\pi^\tau \Sym_\pi^\tau  \de_\pi^\tau D_\pi^\tau b^\tau \mid b, b' \in \BB_\pi\} \\
      &= \Z\text{-span}\{b' D_\pi \de_\pi  \Sym_\pi D_\pi b^\tau \mid b, b' \in \BB_\pi\}
      = I_\pi'.
  \end{align*}

By Proposition~\ref{LGenBasis}, we have $I_\pi=\oplus_{\si \geq \pi} I_\si'$, and we have a nested family of ideals $(I_\pi)_{\pi\in\Pi(\al)}$. To check that $R_\al$ is affine cellular, we need to verify that 
$\bar I_\pi:=I_\pi/I_{>\pi}$ is an affine cell ideal in $\bar R_\al:=R_\al/I_{>\pi}$. As usual we denote $\bar x:=x+I_{>\pi}\in \bar R_\al$ for $x\in R_\al$. 

The affine algebra $B$ in the definition of a cell ideal will be the algebra $\Sym_\pi$, with the automorphism $\si$ being the identity map. The $\Z$-module $V$ will be the formal free $\Z$-module $V_\pi$ on the basis $\BB_\pi$. By  Corollary~\ref{CSymmGen}(i) and Proposition~\ref{LGenBasis}, the following maps are isomomorphisms of $\Sym_\pi$-modules.
\begin{align*}
\eta_\pi&: V_\pi \otimes_\Z \Sym_\pi \to \De(\pi), b \otimes f \mapsto \bar b \bar f \bar e_\pi,\\
\eta'_\pi&: \Sym_\pi \otimes_\Z V_\pi \to \De'(\pi), f \otimes b \mapsto \bar e_\pi \bar f \bar D_\pi \bar b^\tau.
\end{align*}
This allows us to endow $V_\pi \otimes_\Z \Sym_\pi$ with a structure of an $(R_\al,\Sym_\pi)$-bimodule and $\Sym_\pi \otimes_\Z V_\pi$ with a structure of an $(\Sym_\pi,R_\al)$-bimodule.

In view of Corollary~\ref{CSymmGen}(iii),(iv) we see that $\De(\pi)$ (resp. $\De'(\pi)$) is a right (resp. left) $\Sym_\pi$-module, and so we may define an $R_\al$-bimodule homomorphism 
$$\nu_\pi: \De(\pi) \otimes_{\Sym_\pi} \De'(\pi) \to I_\pi / I_{>\pi}, \bar r \bar e_\pi \otimes \bar e_\pi \bar r' \mapsto \bar r \bar e_\pi \bar r'.$$
By Proposition~\ref{LGenBasis}, $\nu_\pi$ is an isomorphism. Let $\mu_\pi:=\nu_\pi^{-1}$. This will be the map $\mu$ in the definition of a cell ideal.

\begin{Theorem}
The above data make $R_\al$ into an affine cellular algebra.
\end{Theorem}
\begin{proof}
To verify that $\bar I_\pi$ is a cell ideal in $\bar R_\al$, we first check that our $(\Sym_\pi,R_\al)$-bimodule structure on $\Sym_\pi\otimes_\Z V_\pi$  
comes from our $(R_\al,\Sym_\pi)$-bimodule structure on $V_\pi \otimes_\Z \Sym_\pi$ via the rule  (\ref{ERightCell}). 
Let $\tw_\pi: V_\pi \otimes_\Z \Sym_\pi \iso \Sym_\pi \otimes_\Z V_\pi$ be the swap map. This is equivalent to the fact that the composition map 
\begin{equation}\label{EDeSwap}
\phi: \De'(\pi) \stackrel{(\eta'_\pi)^{-1}}{\longrightarrow} \Sym_\pi \otimes_\Z V_\pi \stackrel{\tw_\pi^{-1}}{\rightarrow} V_\pi \otimes_\Z \Sym_\pi \stackrel{\eta_\pi}{\to} \De(\pi) = \De(\pi)^\tau,
\end{equation}
is an isomorphism of right $R_\al$-modules. We already know that this is an isomorphism of $\Z$-modules, and so it suffices to check that 
$$
\phi(\bar e_\pi\bar f \bar D_\pi \bar c^\tau \bar r)=\bar r^\tau \phi(\bar e_\pi\bar f \bar D_\pi \bar c^\tau)
$$
for all $f \in \Sym_\pi$, $c \in \BB_\pi$, and $r\in R_\al$. Note that 
$\phi(\bar e_\pi\bar f \bar D_\pi \bar c^\tau)=\bar c\bar f\bar e_\pi$. So we have to check 
\begin{equation}\label{E120613}
\phi(\bar e_\pi\bar f \bar D_\pi \bar c^\tau \bar r)=\bar r^\tau\bar c\bar f\bar e_\pi.
\end{equation}
By Proposition~\ref{LGenBasis}(ii) we can find $\{f_b \mid b \in \BB_\pi\} \subseteq \Sym_\pi$ such that
\begin{equation}\label{EAction}
\bar e_\pi\bar f \bar D_\pi \bar c^\tau \bar r = \sum_{b \in \BB_\pi}{\bar e_\pi \bar f_b \bar D_\pi \bar b^\tau}.
\end{equation}
Also, by Corollary~\ref{CSymmGen}(i), we have 
$$
\bar e_\pi\bar f \bar D_\pi \bar c^\tau = \bar f \bar e_\pi\bar D_\pi \bar c^\tau= \bar e_\pi\bar D_\pi \bar f \bar c^\tau
$$
Using this and the $\tau$-invariance of $D_\pi$ and $\de_\pi$, we  get (\ref{E120613}) as follows: 
  \begin{align*}
   \bar r^\tau\bar c\bar f\bar e_\pi &= \bar r^\tau\bar c\bar f \bar e_\pi^2
    = \bar r^\tau\bar c\bar f \bar D_\pi^\tau \bar \de_\pi^\tau \bar D_\pi^\tau \bar \de_\pi^\tau
    = (\bar \de_\pi \bar D_\pi \bar \de_\pi \bar D_\pi  \bar f \bar c^\tau  \bar r)^\tau\\
    &= (\bar \de_\pi \bar e_\pi \bar D_\pi  \bar f \bar c^\tau  \bar r)^\tau
    = (\bar \de_\pi \bar e_\pi \bar f \bar D_\pi   \bar c^\tau  \bar r)^\tau
    = (\bar \de_\pi \sum_{b \in \BB_\pi} \bar e_\pi \bar f_b \bar D_\pi \bar b^\tau)^\tau\\
    &= (\sum_{b \in \BB_\pi} \bar \de_\pi \bar D_\pi \bar \de_\pi \bar D_\pi \bar f_b \bar b^\tau)^\tau
    = \sum_{b \in \BB_\pi} \bar b \bar f_b \bar D_\pi \bar \de_\pi \bar D_\pi \bar \de_\pi
    = \sum_{b \in \BB_\pi} \bar b \bar f_b \bar e_\pi^2\\
    &= \sum_{b \in \BB_\pi} \bar b \bar f_b \bar e_\pi,
  \end{align*}
  which equals the left hand side of (\ref{E120613}) by definition of $\phi$. 

To complete the proof, it remains to verify the commutativity of (\ref{ECellCD}). This is equivalent to 
$$
 \tau \circ \nu_\pi \circ (\eta_\pi \otimes \eta'_\pi) ((b \otimes f) \otimes (f' \otimes b'))=\nu_\pi \circ (\eta_\pi \otimes \eta'_\pi) ((b' \otimes f') \otimes (f \otimes b))
$$
for all $b,b'\in \BB_\pi$ and $f,f'\in\Sym_\pi$. 
The left hand side equals 
\begin{align*}
&\tau \circ \nu_\pi(\bar b \bar f \bar e_\pi \otimes \bar e_\pi \bar f' \bar D_\pi (\bar b')^\tau)= \tau(\bar b \bar f \bar e_\pi \bar f' \bar D_\pi (\bar b')^\tau)
= \tau(\bar b \bar e_\pi \bar f \bar f' \bar D_\pi (\bar b')^\tau)\\
    = &\  \bar b' \bar D_\pi \bar f' \bar f \bar e_\pi^\tau \bar b^\tau
    = \bar b' \bar D_\pi \bar f' \bar f \bar \de_\pi \bar D_\pi \bar b^\tau
    = \bar b' \bar D_\pi \de_\pi \bar f' \bar f \bar \bar D_\pi \bar b^\tau
    = \bar b' \bar e_\pi \bar f' \bar f \bar D_\pi \bar b^\tau
    \\
    = &\ \bar b' \bar f' \bar e_\pi \bar f \bar D_\pi \bar b^\tau
    = \nu_\pi(\bar b' \bar f' \bar e_\pi \otimes \bar e_\pi \bar f \bar D_\pi \bar b^\tau),
\end{align*}
which equals $\nu_\pi \circ (\eta_\pi \otimes \eta'_\pi) ((b' \otimes f') \otimes (f \otimes b))$, as required. 
\end{proof}

\section{Verification of the Hypothesis}\label{SVerif}

In this section we verify Hypothesis~\ref{HProp} for all finite types. In ADE types (with one exception) this can be do using the theory of homogeneous representations developed in \cite{KRhomog}. This theory is reviewed in the next subsection. We use the cuspidal modules of \cite{HMM}.

Throughout the section $\be$ is a positive root, and $\bar R_\be:=R_\be/I_{>(\be)}$, $\bar r:=r+I_{>(\be)}$ for $r\in R_\be$. 

\subsection{Homogeneous representations} \label{SSHomog}
In this section we assume that the Cartan matrix $A$ is symmetric. In this subsection we fix $\al\in Q_+$ with $d=\height(\al)$. A graded $R_\al$-module is called {\em homogeneous} if it is concentrated in one degree. 
Let  $\bi\in \words_\al$. We call $s_r\in S_d$ an {\em admissible transposition} for $\bi$ if $a_{i_r, i_{r+1}}=0$. 
The {\em word graph} $G_\al$ is the graph with the set of vertices $\words_\al$, and with $\bi,\bj\in \words_\al$ connected by an edge if and only if $\bj=s_r \bi$ for some admissible transposition $s_r$ for $\bi$. 
A connected component $C$ of $G_\al$ is called {\em homogeneous}  
if for some $\bi=(i_1,\dots,i_d)\in C$ the following condition holds:
\begin{equation}\label{ENC}
\begin{split}
\text{if $i_r=i_s$ for some $r<s$ then there exist $t,u$}
\\  
\text{such that $r<t<u<s$ and $a_{i_r,i_t}=a_{i_r,i_u}=-1$.}
\end{split}
\end{equation}

\begin{Theorem}\label{Thomog} {\rm \cite[Theorems 3.6, 3.10, (3.3)]{KRhomog}} 
Let $C$ be a homogeneous connected component of $G_\al$. Let $L(C)$ be the vector space  concentrated in degree $0$ with basis $\{v_\bi\mid \bi\in C\}$ labeled by the elements of $C$. 
The formulas
\begin{align*}
1_\bj v_\bi&=\de_{\bi,\bj}v_\bi \qquad (\bj\in \words_\al,\ \bi\in C),\\
y_r v_\bi&=0\qquad (1\leq r\leq d,\ \bi\in C),\\
\psi_rv_{\bi}&=
\left\{
\begin{array}{ll}
v_{s_r\bi} &\hbox{if $s_r\bi\in C$,}\\
0 &\hbox{otherwise;}
\end{array}
\right.
\quad(1\leq r<d,\ \bi\in C)
\end{align*}
define an action of $R_\al$ on $L(C)$, under which $L(C)$ is a homogeneous irreducible $R_\al$-module. 
Furthermore, $L(C)\not\cong L(C')$ if $C\neq C'$, and every homogeneous irreducible $R_\al$-module, up to a degree shift, is isomorphic to one of the modules $L(C)$. 
\end{Theorem}

We need to push the theory of homogeneous modules a little further. In Proposition~\ref{PHomGenRel} below we give a presentation for a homogeneous module as a cyclic modules generated by a word vector. Let $C$ be a homogeneous component of $G_\al$  and $\bi\in C$. An element $w\in \Si_d$ is called {\em $\bi$-admissible} if it can be written as $w=s_{r_1}\dots s_{r_b}$, where $s_{r_a}$ is an admissible transposition for $s_{r_{a+1}}\dots s_{r_b}\bi$ for all $a=1,\dots,b$. We denote the set of all $\bi$-admissible elements by $\DD_\bi$.

\begin{Lemma} \label{LAdm}
Let $C$ be a homogeneous component of $G_\al$  and $\bi\in C$. Then $\{\psi_wv_\bi\mid w\in\DD_\bi\}$ is a basis of $L(C)$.
\end{Lemma}
\begin{proof}
Note that if $w,w'$ are admissible elements, then $w=w'$ if and only if $w\bi=w'\bi$. Indeed, it suffices to prove that $w\bi=\bi$ implies $w=1$, which follows from the property (\ref{ENC}). The lemma follows. 
\end{proof}

\begin{Proposition}\label{PHomGenRel}
Let $C$ be a homogeneous component of $G_\al$  and $\bi\in C$. Let $J(\bi)$ be the left ideal  of $R_\al$ generated by 
\begin{align}\label{EJ(C)}
\{y_r,1_\bj,\psi_w1_\bi\mid\, 
1\leq r\leq d,\ \bj\in \words_\al\setminus \bi,\ 
w\in\Si_d\setminus \DD_\bi\}. 
\end{align}
Then $R_\al/J_\al\simeq L(C)$ as (graded) left $R_\al$-modules.  
\end{Proposition}
\begin{proof}
Note that the elements in (\ref{EJ(C)}) annihilate the vector $v_\bi\in L(C)$, which generates $L(C)$, whence we have a (homogeneous) surjection 
$$
R_\al/J_\al\onto L(C),\ h+J_\al\mapsto hv_\bi.
$$
To prove that this surjection is an isomorphism it suffices to prove 
that the dimension of $R_\al/J_\al$ is at most $\dim L(C)=|C|$, which follows easily from Lemma~\ref{LAdm}. 
\end{proof}

\subsection{Special Lyndon orders}\label{SSSLO}
Recall the theory of standard modules reviewed in \S\ref{SSSMT}. 
We now specialize to the case of a {\em Lyndon}\, convex order on $\Phi_+$ as studied in \cite{KR}. 
For this we first need to fix a total order `$\leq$' on $I$. This gives rise to a lexicographic order `$\leq$' on the set $\words$. In particular, each finite dimensional $R_\al$-module  has its (lexicographically) highest word, and the highest word of an irreducible module determines the irreducible module uniquely up to an isomorphism. This leads to the natural notion of {\em dominant words} (called good words in \cite{KR}), namely the elements of $\words_\al$ which occur as highest words of finite dimensional $R_\al$-modules. 

The dominant  words of cuspidal modules are characterized among all dominant words  by the property that they are {\em Lyndon words}, so we refer to them as {\em dominant Lyndon words}. There is an explicit bijection
$$
\Phi_+\to\{\text{dominant Lyndon words}\},\ \be\mapsto \bi_\be,
$$
uniquely determined by the property $|\bi_\be|=\be$. Note that this notation $\bi_\be$ will be consistent with the same notation used in \S\ref{SSSSWId}.  

Setting $\be\leq\ga$ if and only if $\bi_\be\leq\bi_\ga$ for $\be,\ga\in\Phi_+$ defines 
a total order on $\Phi_+$ called a {\em Lyndon order}. It is known that each Lyndon order is convex, and the theory of standard modules for Lyndon orders, developed in \cite{KR}, fits into the general theory described in \S\ref{SSSMT}. However, working with Lyndon orders  allows us to be a little more explicit. In particular, given a a root partition $\pi=(p_1,\dots,p_N)\in\Pi(\al)$, set 
\begin{equation}\label{EIPi}
\bi_\pi:=\bi_{\be_1}^{p_1}\dots\bi_{\be_N}^{p_N}\in\words_\al.
\end{equation}

\begin{Lemma} \label{LHighestWt}%
{\rm \cite[Theorem 7.2]{KR}} 
Let $\pi\in\Pi(\al)$. Then $\bi_\pi$ is the highest word of $L(\pi)$. 
\end{Lemma}

From now on, we fix the notation for the Dynkin diagrams as follows:
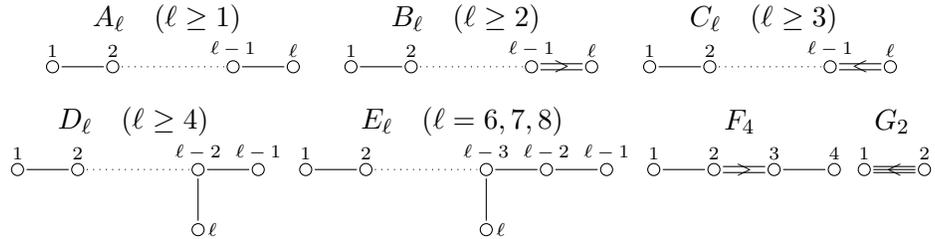
\[
\begin{dynkin}
  \draw (.5,.8)node[anchor=west,font=\normalsize]{$A_\ell\quad(\ell \geq 1)$};
  \draw (0,0)node[above]{$1$} circle (0.10);
  \draw
  (0.15,0)--(0.85,0);
  \draw (1,0)node[above]{$2$} circle (0.10);
  \draw[
  dotted] (1.15,0)--(2.85,0);
  \draw (3,0)node[above]{$\ell-1$} circle (0.10);
  \draw
  (3.15,0)--(3.85,0);
  \draw (4,0)node[above]{$\ell$} circle (0.10);
\end{dynkin}\quad
\begin{dynkin}
  \draw (.5,.8)node[anchor=west,font=\normalsize]{$B_\ell\quad(\ell \geq 2)$};
  \draw (0,0)node[above]{$1$} circle (0.10);
  \draw
  (0.15,0)--(0.85,0);
  \draw (1,0)node[above]{$2$} circle (0.10);
  \draw[
  dotted] (1.15,0)--(2.85,0);
  \draw (3,0)node[above]{$\ell-1$} circle (0.10);
  \draw (3.15,.05)--(3.85,.05);
  \draw (3.15,-.05)--(3.85,-.05);
  \draw (3.5,0)node{$>$};
  \draw (4,0)node[above]{$\ell$} circle (0.10);
\end{dynkin}\quad
\begin{dynkin}
  \draw (.5,.8)node[anchor=west,font=\normalsize]{$C_\ell\quad(\ell \geq 3)$};
  \draw (0,0)node[above]{$1$} circle (0.10);
  \draw
  (0.15,0)--(0.85,0);
  \draw (1,0)node[above]{$2$} circle (0.10);
  \draw[
  dotted] (1.15,0)--(2.85,0);
  \draw (3,0)node[above]{$\ell-1$} circle (0.10);
  \draw (3.15,.05)--(3.85,.05);
  \draw (3.15,-.05)--(3.85,-.05);
  \draw (3.5,0)node{$<$};
  \draw (4,0)node[above]{$\ell$} circle (0.10);
\end{dynkin}
\]
\[
\begin{dynkin}
  \draw (.5,.8)node[anchor=west,font=\normalsize]{$D_\ell\quad(\ell \geq 4)$};
  \draw (0,0)node[above]{$1$} circle (0.10);
  \draw
  (0.15,0)--(0.85,0);
  \draw (1,0)node[above]{$2$} circle (0.10);
  \draw[
  dotted] (1.15,0)--(2.85,0);
  \draw (3,0)node[above]{$\ell-2$} circle (0.10);
  \draw
  (3,-.15)--(3,-.85);
  \draw (3,-1)node[right]{$\ell$} circle (0.10);
  \draw
  (3.15,0)--(3.85,0);
  \draw (4,0)node[above]{$\ell-1$} circle (0.10);
\end{dynkin}
\begin{dynkin}
  \draw (.75,.8)node[anchor=west,font=\normalsize]{$E_\ell\quad(\ell = 6,7,8)$};
  \draw (0,0)node[above]{$1$} circle (0.10);
  \draw
  (0.15,0)--(0.85,0);
  \draw (1,0)node[above]{$2$} circle (0.10);
  \draw[
  dotted] (1.15,0)--(2.85,0);
  \draw (3,0)node[above]{$\ell-3$} circle (0.10);
  \draw
  (3,-.15)--(3,-.85);
  \draw (3,-1)node[right]{$\ell$} circle (0.10);
  \draw
  (3.15,0)--(3.85,0);
  \draw (4,0)node[above]{$\ell-2$} circle (0.10);
  \draw
  (4.15,0)--(4.85,0);
  \draw (5,0)node[above]{$\ell-1$} circle (0.10);
\end{dynkin}
\begin{dynkin}
  \draw (1,.8)node[anchor=west,font=\normalsize]{$F_4$};
  \draw (0,0)node[above]{$1$} circle (0.10);
  \draw (0.15,0)--(0.85,0);
  \draw (1,0)node[above]{$2$} circle (0.10);
  \draw (1.15,.05)--(1.85,.05);
  \draw (1.15,-.05)--(1.85,-.05);
  \draw (1.5,0)node{$>$};
  \draw (2,0)node[above]{$3$} circle (0.10);
  \draw (2.15,0)--(2.85,0);
  \draw (3,0)node[above]{$4$} circle (0.10);
\end{dynkin}
\begin{dynkin}
  \draw (0,.8)node[anchor=west,font=\normalsize]{$G_2$};
  \draw (0,0)node[above]{$1$} circle (0.10);
  \draw (0.15,.05)--(0.85,.05);
  \draw (0.15,0)--(0.85,0);
  \draw (0.15,-.05)--(0.85,-.05);
  \draw (0.5,0)node{$<$};
  \draw (1,0)node[above]{$2$} circle (0.10);
\end{dynkin}
\]
Also, we choose the signs $\eps_{ij}$  as in \S\ref{SSKLR} and the total order $\leq$ on $I$ so that $\eps_{ij}=1$  and  $i< j$ if the corresponding labels $i$ and $j$ satisfy $i< j$ as integers. 

\subsection{Homogeneous roots}
We stick with the choices made in \S\ref{SSSLO}. Throughout the subsection, we assume that the Cartan matrix is of $ADE$ type and $\be \in \Phi_+$ is such that $\bi_\be$ is homogeneous. Let $d:=\height(\be)$. The module $L(\be)$ is concentrated in degree 0, and each of its word spaces is one dimensional. Set $\DD_\be := \DD_{\bi_\be}$. Then we can take $\BB_\be = \{\psi_w e(\bi_\be) \mid w \in \DD_\be\}$. Let $\de_\be = D_\be = e(\bi_\be)$, and define $y_\be := y_d e(\bi_\be)$. All parts of Hypothesis~\ref{HProp} are trivially satisfied, except (v). In the rest of this subsection we verify Hypothesis~\ref{HProp}(v).

\begin{Lemma}\label{LBadInBlock}
Let $w \in \Si_d \setminus \DD_{\be}$. Then $\psi_w \Pol_{d} e(\bi_{\be})  \subseteq I_{>(\be)}.$
\end{Lemma}

\begin{proof}
We have $\psi_w=\psi_{r_1} \dots \psi_{r_m}$ for a reduced decomposition $w=s_{r_1} \dots s_{r_m}$. Let $k$ be the largest index such that $s_{r_k}$ is not an admissible transposition of $s_{r_{k+1}} \dots s_{r_m} \bi_\be$. By Theorem~\ref{Thomog}, $s_{r_k} \dots s_{r_m} \bi_\be$ is not a word of $L(\be)$. So by Corollary~\ref{LBadWords}, $$\psi_{r_{k}}\dots \psi_{r_m}\Pol_d e( \bi_\be)
=e(s_{r_k} \dots s_{r_m} \bi_\be)\psi_{r_{k}}\dots \psi_{r_m}\Pol_d e( \bi_\be)
\subseteq I_{>(\be)},$$ whence $\psi_w\Pol_d e( \bi_\be)\subseteq I_{>(\be)}$. 
\end{proof}

\begin{Lemma}\label{LyEq}
  Given $1 \leq r,s \leq d$, we have $(y_s - y_r) e(\bi_{\be}) \in I_{>(\be)}$.
\end{Lemma}
\begin{proof}
 We prove by induction on $s=1,\dots,d$ that $(y_s - y_r) e(\bi_\be) \in I_{>(\be)}$ for all $1 \leq r \leq s$. The base case $s=1$ is trivial. Let $s>1$, and 
write $\bi_\be = (i_1, \dots, i_d)$. If $i_r \cdot i_s = 0$ for all $1 \leq r < s$, then 
\[
  (i_s, i_1, i_2, \dots, i_{s-1}, i_{s+1}, \dots, i_d)
\]
is a word of $L(\be)$. On the other hand, Lemma~\ref{LHighestWt} says that $\bi_\be$ is the largest word of $L(\be)$ and so $i_s < i_1$. But then  $\bi_\be$ is not a Lyndon word, which is a contradiction. Thus there exists some $r < s$ with $i_r \cdot i_s \neq 0$. Since the  Cartan matrix is assumed to be of ADE type, either $i_r \cdot i_s = -1$ or $i_r = i_s$. In the second case, by homogeneity (\ref{ENC}) we can find $r < r' < s$ with $i_{r'} \cdot i_s = -1$. This shows that the definition 
$
  t:= \max\{r \mid r < s \text{ and } i_r \cdot i_s = -1\}
$
makes sense. Once again by homogeneity we must have that $i_r \cdot i_s = 0$ for any $r$ with $t < r < s$. Therefore, using defining relations in $R_\al$, we get 
\[
  (\psi_{s-1} \dots \psi_t)(\psi_t \dots \psi_{s-1}) e(\bi_\be) = \pm (y_s - y_t) e(\bi_\be).
\]
On the other hand, the cycle $(t, t+1, \dots, s)$ is not an element of $\DD_\be$. By Lemma~\ref{LBadInBlock} we must have $\psi_t \dots \psi_{s-1} e(\bi_\be) \in I_{>(\be)}$. This shows that $(y_s - y_t) e(\bi_\be) \in I_{>(\be)}$, and therefore by induction that $(y_s - y_r) e(\bi_\be) \in I_{>(\be)}$ for every $r$ with $1 \leq r \leq s$.
\end{proof}

Recall the notation $\bar R_\be:= R_\be/I_{>(\be)}$ and $\bar r:=r+I_{>(\be)}\in\bar R_\be$ for $r\in R_\be$.

\begin{Corollary}
We have  that $\bar e_\be \bar R_\be \bar e_\be$ is generated by $\bar y_\be$.
\end{Corollary}
\begin{proof}
By Theorem~\ref{TBasis}, an element of $e_\be R_\be e_\be$ is a linear combination of terms of the form $\psi_w y_1^{a_1} \dots y_d^{a_d} e(\bi_\be)$ such that $w\bi_\be=\bi_\be$. If $w \notin \DD_\be$, then $\psi_w e_\be \in I_{>(\be)}$ by Lemma~\ref{LBadInBlock}. Otherwise, Lemma~\ref{LAdm} shows that $w=1$. Therefore, $\bar e_\be \bar R_\be \bar e_\be$ is spanned by terms of the form $\bar y_1^{a_1} \dots \bar y_d^{a_d} \bar e_\be$. In view of Lemma~\ref{LyEq}, we see that $\bar e_\be \bar R_\be \bar e_\be$ is generated by $\bar y_\be=\bar y_d$.
\end{proof}

\subsection{Types $ADE$}\label{SSADE}
Throughout the subsection, we assume again that the Cartan matrix is of $ADE$ type. By \cite{HMM}, with a correction made in \cite[Lemma A7]{BKM}, if $\be \in \Phi_+$ is any positive root, except the highest root in type $E_8$, then $\bi_\be$ is homogeneous. We have proved in the previous subsection that Hypothesis~\ref{HProp} holds in this case.

Now, we deal with the highest root $$\theta := 2\al_1 + 3\al_2 + 4\al_3 + 5\al_4 + 6\al_5 + 4\al_6 + 2\al_7 + 3\al_8$$ 
in type $E_8$. 
By \cite[Example A.5]{BKM}, the corresponding Lyndon word is 
$$\bi_\theta = 12345867564534231234586756458.$$
Define the positive roots
\begin{align}
\theta_1&:=\al_1 + \al_2 + \al_3 + 2\al_4 + 3\al_5 + 2\al_6 + \al_7 + 2\al_8,
\label{ETheta1}
\\
{\theta_2}&:=\al_1 + 2\al_2 + 3\al_3 + 3\al_4 + 3\al_5 + 2\al_6 + \al_7 + \al_8.
\label{ETheta2}
\end{align}
Then the root partition $(\theta_1,{\theta_2})$ is a minimal element of $\Pi(\theta)\setminus \{(\theta)\}$. Moreover, $\bi_{\theta_2} = 1234586756453423$ and $\bi_{\theta_1} = 1234586756458$. Indeed, one sees by inspection that these words are highest words in the corresponding homogeneous representations and are Lyndon. 
Finally, we have  $\bi_\theta = \bi_{\theta_2} \bi_{\theta_1}$. 

Denote by $v_{\theta_1}$ and $v_{\theta_2}$ non-zero vectors in the $\bi_{\theta_1}$- and $\bi_{\theta_2}$-word spaces in the homogeneous modules $L({\theta_1})$ and $L({\theta_2})$, respectively. Note that $L({\theta_1})\boxtimes L({\theta_2})$ is naturaly a submodule of $L({\theta_1})\circ L({\theta_2})$, so we can consider $v_{\theta_1}\otimes v_{\theta_2}$ as a cyclic vector of $L({\theta_1})\circ L({\theta_2})$, and similarly 
$v_{\theta_2}\otimes v_{\theta_1}$ as a cyclic vector of $L({\theta_2})\circ L({\theta_1})$. By definition, $L({\theta_1})\circ L({\theta_2})$ is the proper standard module $\bar\De({\theta_1},{\theta_2})$, and let $v_{{\theta_1},{\theta_2}}$ be the image of $v_{\theta_1}\otimes v_{\theta_2}$ under the natural projection $\bar\De({\theta_1},{\theta_2})\onto L({\theta_1},{\theta_2})$. 
Denote by $w(\theta)$ the element of $\Si_{29}$ which sends $(1,\dots,29)$ to $(17,\dots,29,1,\dots,16)$.
The following has been established in \cite{BKM}, see especially \cite[Theorem A.9, Proof]{BKM}, but we sketch its very easy proof for the reader's convenience. 

\begin{Lemma}\label{LSES} 
The multiplicity of the highest word $\bi_\theta$ in $L(\theta)$ is one. Moreover, there is a non-zero vector $v_\theta$ in the $\theta$-word space of $L(\theta)$ and  homogeneous $R_\theta$-module maps 
\begin{align*}
\mu&: L({\theta_1},{\theta_2})\langle 1\rangle \to L({\theta_2})\circ L({\theta_1}),\ v_{{\theta_1},{\theta_2}}\mapsto \psi_{w(\theta)}(v_{\theta_2}\otimes v_{\theta_1}),
\\
\nu&: L({\theta_2}) \circ L({\theta_1}) \to L(\theta),\ v_{\theta_2}\otimes v_{\theta_1}\mapsto v_\theta,
\end{align*}
such that the sequence 
\[
  0 \to L({\theta_1},{\theta_2})\langle 1\rangle \stackrel{\mu}{\to} L({\theta_2}) \circ L({\theta_1}) \stackrel{\nu}\to L(\theta) \to 0
\]
is exact. Finally,  
$$\CH L(\theta)=(\CH L({\theta_2})\circ \CH L({\theta_1}) - q\,\CH L({\theta_1})\circ \CH L({\theta_2}))/(1-q^2).$$ 
\end{Lemma}
\begin{proof}
By \cite[Theorem 7.2(ii)]{KR}, the multiplicity of the word $\bi_{\theta_1}\bi_{\theta_2}$ in $L({\theta_1})\circ L({\theta_2})$ is $1$. Moreover, an explicit check shows that the multiplicity of $\bi_\theta$ in $L({\theta_1})\circ L({\theta_2})$ is $q$. We conclude using Theorem~\ref{TStand}  
and the minimality of $({\theta_1},{\theta_2})$ in $\Pi(\theta)\setminus\{(\theta)\}$ that 
the standard module $L({\theta_1})\circ L({\theta_2})$ is uniserial with head $L({\theta_1},{\theta_2})$ and socle $L(\theta)\langle 1\rangle$. 
The result follows from these observations since $L({\theta_1},{\theta_2})$ is $\circledast$-self-dual and 
$(L({\theta_1})\circ L({\theta_2}))^\circledast \simeq L({\theta_2})\circ L({\theta_1})\langle -1\rangle$ in view of \cite[Theorem 2.2]{LV}.
\end{proof}

Consider the parabolic subgroup $\Si_{\height(\theta_2)}\times\Si_{\height(\theta_1)}\subseteq \Si_d$ and define
$$
\DD_{\theta_2,\theta_1}:=\{(w_2,w_1)\in \Si_{\height(\theta_2)}\times\Si_{\height(\theta_1)}\mid w_2\in\DD_{\theta_2},\ w_1\in\DD_{\theta_1}\}.
$$
With this notation we finally have: 

\begin{Lemma}\label{LIrrE8}
The cuspidal module $L(\theta)$ is generated by a degree $0$ vector $v_\theta$ subject only to the relations:
\begin{align}
\label{Rel1}
  (e(\bj)-\de_{\bj, \bi_\theta}) v_\theta &=0, \quad \textup{for all $\bj\in \words_\theta$},\\ \label{Rel2}
  y_r v_\theta &= 0, \quad \textup{for all $r=1,\dots,\height(\theta)$},\\ \label{RE_81}
  \psi_w v_\theta&=0, \quad \textup{for all $w\in (\Si_{\height(\theta_2)}\times\Si_{\height(\theta_1)})\setminus \DD_{\theta_2,\theta_1}$},\\
  \psi_{w(\theta)}v_\theta&=0.
\end{align}
\end{Lemma}
\begin{proof}
The theorem  follows easily from Proposition~\ref{PHomGenRel} applied to homogeneous modules $L({\theta_1})$ and $L({\theta_2})$, and Lemma~\ref{LSES}. 
\end{proof}

We now define $\de_\theta = D_\theta = e(\bi_\theta)$, and $y_\theta = y_{\height(\theta)} e(\bi_\theta)$. All parts of Hypothesis~\ref{HProp} are trivially satisfied, except (v). We now verify Hypothesis~\ref{HProp}(v).

\begin{Lemma}\label{PParaE8}
We have
\[
  \iota_{\theta_2, \theta_1}(I_{>(\theta_2)} \otimes R_{\theta_1} + R_{\theta_2} \otimes I_{>(\theta_1)}) \subseteq I_{>(\theta)}.
\]
\end{Lemma}
\begin{proof}
Apply Proposition~\ref{LParaIdeal1New} twice with $m=2$, $\ga_1=\theta_2$, $\ga_2=\theta_1$, $\pi=(\theta)$, and either $k=1$ and $\pi_0=(\theta_2)$, or $k=2$ and $\pi_0=(\theta_1)$. 
\end{proof}


\begin{Lemma}
We have that  $\bar e_\theta \bar R_\theta \bar e_\theta$ is generated by $\bar y_\theta$.
\end{Lemma}
\begin{proof}
By Theorem~\ref{TBasis}, an element of $e_\theta R_\theta e_\theta$ is a linear combination of terms of the form $\psi_w y_1^{a_1} \dots y_d^{a_d} e(\bi_\theta)$ such that $w\bi_\theta=\bi_\theta$. If $w\in (\Si_{\height(\theta_2)}\times\Si_{\height(\theta_1)})\setminus \DD_{\theta_2,\theta_1}$, then $\psi_w e_\theta \in I_{>(\theta)}$ by Lemmas~\ref{PParaE8} and~\ref{LBadInBlock}. So we may assume that  $w = uv$ with $u \in \Si^{\height(\theta_2),\height(\theta_1)}$, $v \in \DD_{\theta_2, \theta_1}$. It is easy to check that the only such permutation that fixes $\bi_\theta$ is the identity. We therefore see that $\bar e_\theta \bar R_\theta \bar e_\theta$ is generated by $\bar y_1, \dots, \bar y_{\height(\theta)}$.

Note that $\height(\theta_2)=16$ and $\height(\theta_1)=13$. Using the cases $\be=\theta_2$ and $\be=\theta_1$ proved above and Lemma~\ref{PParaE8}, we have that $(y_r - y_s)e(\bi_\theta) \in I_{>(\theta)}$ if $1 \leq r,s \leq 16$ or $17 \leq r,s \leq 29$. It remains to show that $(y_r - y_s)e(\bi_\theta) \in I_{>(\theta)}$ for some $1 \leq r \leq 16$ and $17 \leq s \leq 29$. 
Let $w \in \Si_{29}$ be the cycle $(27, 26, \dots, 16)$. By considering words and using Corollary~\ref{LBadWords}, one can verify that
\begin{align*}
  \psi_w^\tau \psi_w e(\bi_\theta) &\equiv (y_{16}  - y_{27})e(\bi_\theta) \pmod{I_{>(\theta)}}.
\end{align*}
On the other hand, by the formula for the character of $L(\theta)$ from Lemma~\ref{LSES}, we have that $w \bi_\theta$ is not a word of $L(\theta)$. Therefore, by Corollary~\ref{LBadWords}, we have that $\psi_w e(\bi_\theta) \in I_{>(\theta)}$, so $(y_{16} - y_{27})e(\bi_\theta) \in I_{>(\theta)}$, and we are done.
\end{proof}

\subsection{Non-symmetric types}\label{SSBCFG}
Now we deal with non-symmetric Cartan matrices, i.e. Cartan matrices of $BCFG$ types. 

\begin{Lemma}\label{LHelper}
  Suppose that $\de_\be, D_\be \in e(\bi_\be) R_\be e(\bi_\be)$ have been chosen so that Hypothesis~\ref{HProp}(iii) is satisfied. If the minimal degree component of $e(\bi_\be) R_\be e(\bi_\be)$ is spanned by $D_\be$, then Hypothesis~\ref{HProp}(i) and (vi) are satisfied.
\end{Lemma}
\begin{proof}
Since $D_\be \de_\be D_\be$ has the same degree as $D_\be$, the assumption above implies that $D_\be \de_\be D_\be$ is proportional to $D_\be$. Acting on $v_\be^+$ and using Hypothesis~\ref{HProp}(iii) gives $D_\be \de_\be D_\be = D_\be$, which upon multiplication by $\de_\be$ on the right gives the property $e_\be^2=e_\be$, which is even stronger than (i).

To see (vi), we look at the lowest degree component in $e(\bi_\be\bi_\be)R_{2\be}e(\bi_\be\bi_\be)$ using \cite[Lemma 5.3(ii)]{KR} and commutation relations in the algebra $R_{2\be}$.
\end{proof}

It will be clear in almost all cases that the condition of Lemma~\ref{LHelper} will be satisfied, and moreover Hypothesis~\ref{HProp}(ii) and (iv) are easy to verify by inspection. This leaves Hypothesis~\ref{HProp}(v) to be shown in each case.

\subsubsection{Type $B_l$}\label{SSTypeB}
The set of positive roots is broken into two types. For $1 \leq i \leq j \leq l$ we have the root $\al_i + \dots + \al_j$, and for $1 \leq i < j \leq l$ we have the root $\al_i + \dots + \al_{j-1} + 2\al_j + \dots + 2\al_l$.

Let $\be := \al_i + \dots + \al_j$. Then $\bi_{\be} := (i, \dots, j)$, and the irreducible module $L({\be})$ is one-dimensional with character $\bi_{\be}$. Define $\de_{\be} := D_{\be} := e(\bi_{\be})$ and $y_{\be} := y_{d} e(\bi_\be)$. Using Corollary~\ref{LBadWords} one sees that $\psi_r e_{\be} \in I_{>(\be)}$ for all $r$, which by Theorem~\ref{TBasis} shows that $\bar R_\be \bar e_\be = F[\bar y_1, \dots, \bar y_d] \bar e_\be$. This also shows that for $1 \leq r \leq d$ we have the elements of $I_{>(\be)}$:
\[
  \psi_r^2 e_{\be} = \begin{cases} (y_r-y_{r+1}^2)e_{\be}, & \text{if $j=l$ and $r=d-1$}\\
        (y_r - y_{r+1})e_{\be}, & \text{otherwise.}
        \end{cases}
\]
It follows that $\bar R_\be \bar e_\be = F[\bar y_\be] \bar e_\be$, and thus $\bar e_{\be} \bar R_{\be} \bar e_{\be}$ is generated by $\bar e_{\be} \bar y_{\be} \bar e_{\be}$.

Consider $\be := \al_i + \dots + \al_{j-1} + 2\al_j + \dots + 2\al_l$. In this case, $\bi_\be = (i, \dots, l, l, \dots, j)$, and $\CH L({\be}) = (q+q^{-1}) \bi_{\be}$. Define $\de_{\be} := y_{l-i+2} e(\bi_{\be})$, $D_{\be} := \psi_{l-i+1} e(\bi_{\be})$, and $y_{\be} = y_1 e(\bi_\be)$. Using Corollary~\ref{LBadWords}, one sees that $\psi_r e(\bi_{\be}) \in I_{>(\be)}$ for $r \neq l-i+1$. It is also clear that $\psi_{l-i+1} e_{\be} = 0$, and therefore by Theorem~\ref{TBasis}, $\bar R_{\be} \bar e_{\be} = F[\bar y_1, \dots, \bar y_d] e_{\be}$. We also have the following elements of $I_{>(\be)}$:
\[
\psi_r^2 e(\bi_{\be}) = \begin{cases} (y_r - y_{r+1})e(\bi_{\be}), & \text{for $1 \leq r \leq l-i-1$;}\\
  (y_{l-i} - y_{l-i+1}^2)e(\bi_{\be}), & \text{for $r = l-i$};\\
  (y_{l-i+3} - y_{l-i+2}^2)e(\bi_{\be}), & \text{for $r = l-i+2$};\\  
  (y_{r+1} - y_r)e(\bi_{\be}) & \text{for $l-i+3 \leq r \leq d-1$}.
\end{cases}
\]
Taken together, these show that $\bar R_{\be} \bar e(\bi_{\be}) = F[\bar y_{l-i+1}, \bar y_{l-i+2}] \bar e(\bi_{\be})$. Multiplying on both sides by $\bar e_{\be}$ and using the KLR / nil-Hecke relations, we have 
$$\bar e_{\be} \bar R_{\be} \bar e_{\be} = F[\bar y_{l-i+1} + \bar y_{l-i+2}, \bar y_{l-i+1} \bar y_{l-i+2}] \bar e_{\be}.$$
Furthermore, $(y_{l-i+1} + y_{l-i+2})\psi_{l-i+1}e(\bi_\be) = \psi_{l-i+1} \psi_{l-i}^2 \psi_{l-i+1} e(\bi_\be) \in I_{>(\be)}$, and so in fact
$$\bar e_{\be} \bar R_{\be} \bar e_{\be} = F[\bar y_{l-i+1}^2] \bar e_{\be} = F[\bar y_1] \bar e_{\be}.$$

\subsubsection{Type $C_l$}
The set of positive roots is broken into three types. For $1 \leq i \leq j \leq l$ we have the root $\al_i + \dots + \al_j$, 
for $1 \leq i < j < l$ we have the root $\al_i + \dots + \al_{j-1} + 2\al_j + \dots + 2\al_{l-1} + \al_l$, 
and for $1 \leq i < l$ we have the root $2\al_i + \dots + 2\al_{l-1} + \al_l$.

Consider $\be = \al_i + \dots + \al_j$. Then $\bi_\be = (i, \dots, j)$ and $\CH L(\be) = \bi_{\be}$. Define $\de_\be = D_\be := e(\bi_\be)$. Define $y_\be := y_1 e(\bi_\be)$.  Using Corollary~\ref{LBadWords} one sees that $\psi_r e_{\be} \in I_{>(\be)}$ for all $r$, which by Theorem~\ref{TBasis} shows that $\bar R_\be \bar e_\be = F[\bar y_1, \dots, \bar y_d] \bar e_\be$. This also shows that for $1 \leq r \leq d$ we have the elements of $I_{>(\be)}$:
\[
  \psi_r^2 e_{\be} = \begin{cases} (y_r^2-y_{r+1})e_{\be}, & \text{if $j=l$ and $r=d-1$}\\
        (y_r - y_{r+1})e_{\be}, & \text{otherwise.}
        \end{cases}
\]
Consequently, $\bar e_{\be} \bar R_{\be} \bar e_{\be}$ is generated by $\bar e_{\be} \bar y_{\be} \bar e_{\be}$.

Consider $\be = \al_i + \dots + \al_{j-1} + 2\al_j + \dots + 2\al_{l-1} + \al_l$. Then $\bi_\be = (i, \dots, l-1, l, l-1, \dots, j)$ and $\CH L(\be) = (i, \dots, l-1, l, l-1, \dots, j)$. Define $\de_\be = D_\be := e(\bi_\be)$, and $y_\be := y_1 e(\bi_\be)$.  Using Corollary~\ref{LBadWords} one sees that $\psi_r e_{\be} \in I_{>(\be)}$ for all $r$, which by Theorem~\ref{TBasis} shows that $\bar R_\be \bar e_\be = F[\bar y_1, \dots, \bar y_d] \bar e_\be$. This also shows that for $1 \leq r \leq d$ we have the elements of $I_{>(\be)}$:
\[
\psi_r^2 e_{\be} = \begin{cases} (y_r - y_{r+1}) e_{\be}, & \text{for $1 \leq r \leq l-i-1$;}\\
  (y_{l-i}^2 - y_{l-i+1})e_{\be}, & \text{for $r = l-i$};\\
  (y_{l-i+2}^2 - y_{l-i+1})e_{\be}, & \text{for $r = l-i+1$};\\  
  (y_{r+1} - y_r)e_{\be} & \text{for $l-i+2 \leq r \leq d-1$}.
\end{cases}
\]
It follows that $\bar R_\be \bar e_\be = F[\bar y_{l-i}, \bar y_{l-i+2}] \bar e_\be$. Furthermore, by the relation (\ref{R7}),
$$
(y_{l-i} + y_{l-i+2}) e_\be = (\psi_{l-i+1} \psi_{l-i} \psi_{l-i+1} - \psi_{l-i} \psi_{l-i+1} \psi_{l-i}) e_\be \in I_{>(\be)}
$$
and therefore $\bar R_\be \bar e_\be = F[\bar y_{l-i}] e_\be = F[\bar y_\be] \bar e_\be$.

Consider $\be = 2\al_i + \dots + 2\al_{l-1} + \al_l$. Then $\bi_{\be} = (i, \dots, l-1, i, \dots, l)$ and 
$$\CH L(\be) = q((i,\dots,l-1) \circ (i,\dots,l-1))\cdot(l).$$
Let $w \in \Si_{d}$ be the permutation that sends $(1,2,\dots, d)$ to $(l-i+1, \dots, d-1, 1, \dots, l-i, d)$, and define $D_{\be} := \psi_w e(\bi_{\be})$. Define also $\de_{\be} := y_{d-1} e(\bi_{\be})$ and $y_{\be} := y_d e(\bi_\be)$. Set $\ga = \al_i + \dots + \al_{l-1}$. 
Since $I_{>(\ga^2)}$ is generated by idempotents $e(\bi)$ with $\bi > \bi_\ga^2$, and $\bi_\be = \bi_\ga^2 i_l$ is the highest weight of $L_\be$, we see that
$$
\iota_{2\ga,\al_l}(I_{>(\ga^2)} \otimes R_{\al_l}) \subseteq I_{>(\be)}.
$$
Let $\mu: \bar R_{2\ga} \boxtimes R_{\al_l} \to \bar R_\be$ be the induced map. Note that every weight of $L(\be)$ ends with $l$, so that $\psi_u e(\bi_\be) \in I_{>(\be)}$ unless $u \in \Si_{d-1,1}$, by Corollary~\ref{LBadWords}. Therefore, applying (\ref{EEndo}) in the type $A$ case of $(\ga^2)$ (which has already been verified), we obtain
$$
\bar e_\be \bar R_\be \bar e_\be = \mu(\bar e_{(\ga^2)} \bar R_{2\ga} \bar e_{(\ga^2)} \otimes R_{\al_l}) = \bar e_\be \O[\bar y_{l-i} + \bar y_{2l-2i}, \bar y_{l-i} \bar y_{2l-2i}, \bar y_d] \bar e_\be.
$$
Furthermore,
$$
(\bar y_{l-i} + \bar y_{2l-2i}) \bar e(\bi_\be) = \bar \psi_{l-i} \dots \bar \psi_{2l-2i-1} \bar \psi_{2l-2i}^2 \bar \psi_{2l-2i-1} \dots \bar \psi_{l-i} \bar e(\bi_\be) = 0
$$
and $(\bar y_{2l-2i}^2 - \bar y_d) \bar e(\bi_\be) = \bar \psi_{2l-2i}^2 \bar e(\bi_\be) = 0.$ Thus $\bar e_\be \bar R_\be \bar e_\be$ is generated by $\bar e_\be \bar y_d \bar e_\be$.

\subsubsection{Type $F_4$}
We write $\be = c_1 \al_1 + c_2 \al_2 + c_3 \al_3 + c_4 \al_4 \in \Phi_+$. If $c_4 = 0$, then this root lies in a subsystem of type $B_3$ with the {\em same order} as in section~\ref{SSTypeB} and we are done. 

If $\be = \al_i + \dots +\al_j$ for some $1\leq i \leq j \leq 4$, then $\bi_\be = (i,\dots,j)$ and $\CH L(\be) = (i,\dots,j)$. In this case we take $D_\be = \de_\be = e(\bi_\be)$, and set $y_\be = y_{\height(\be)} e(\bi_\be)$.

The following table shows the choice of data for the remaining roots, except for the highest root $\be = 2\al_1 + 3\al_2 + 4\al_3 + 2\al_4$, which we discuss separately. In each of these cases, the hypotheses may be verified by employing the same methods used above. For example, in each case either Hypothesis~\ref{HProp}(i)-(iv),(vi) may be verified directly or with the help of Lemma~\ref{LHelper} when it applies.

\begin{tabular}{|c|c|c|c|}
\hline
$\bi_\be$ & $D_\be$ & $\de_\be$ & $y_\be$ \\ \hline
$2343$ & $e(\bi_\be)$ & $e(\bi_\be)$ & $y_3 e(\bi_\be)$ \\
$12343$ & $e(\bi_\be)$ & $e(\bi_\be)$ & $y_5 e(\bi_\be)$ \\
$23434$ & $\psi_3\psi_2\psi_4\psi_3e(\bi_\be)$ & $y_5e(\bi_\be)$ & $y_1e(\bi_\be)$ \\
$123432$ & $e(\bi_\be)$ & $e(\bi_\be)$ & $y_5e(\bi_\be)$ \\
$123434$ & $\psi_4\psi_3\psi_5\psi_4e(\bi_\be)$ & $y_6e(\bi_\be)$ & $y_2e(\bi_\be)$ \\
$1234323$ & $e(\bi_\be)$ & $e(\bi_\be)$ & $y_5 e(\bi_\be)$ \\
$1234342$ & $\psi_4\psi_3\psi_5\psi_4e(\bi_\be)$ & $y_6e(\bi_\be)$ & $y_2e(\bi_\be)$ \\
$12343423$ & $\psi_4\psi_3\psi_5\psi_4e(\bi_\be)$ & $y_6e(\bi_\be)$ & $y_8 e(\bi_\be)$ \\
$123434233$ & $\psi_4\psi_3\psi_5\psi_4\psi_8e(\bi_\be)$ & $y_6y_9e(\bi_\be)$ & $y_7e(\bi_\be)$ \\
$1234342332$ & $\psi_4\psi_3\psi_5\psi_4\psi_8e(\bi_\be)$ & $y_6y_9e(\bi_\be)$ & $y_{10}e(\bi_\be)$ \\
\hline
\end{tabular}

Consider now $\be = 2\al_1 + 3\al_2 + 4\al_3 + 2\al_4$, where $\bi_\be = (12343123432)$. Let $w \in \Si_{11}$ be the permutation that sends $(1,\dots, 11)$ to $(6,7,8,9,10,1,2,3,4,5,11)$, and set $D_\be = \psi_w e(\bi_\be)$. Let $\de_\be = y_{10} e(\bi_\be)$, and $y_\be = y_{11} e(\bi_\be)$. Define $\ga = \al_1 + \al_2 + 2\al_3 + \al_4$. There is a map $\mu: \bar R_{2\ga} \boxtimes R_{\al_2} \to \bar R_\be$. This map is not surjective, but one can show that
$$
\bar e(\bi_\be) \bar R_\be \bar e(\bi_\be) = \mu(\bar e(\bi_\ga^2) \bar R_{2\ga} \bar e(\bi_\ga^2) \otimes R_{\al_2})
$$
and thus
\begin{align*}
  \bar e_\be \bar R_\be \bar e_\be &= \mu(\bar e_{(\ga^2)} \bar R_{2\ga} \bar e_{(\ga^2)} \otimes R_{\al_2}) = \mu(\O[\bar y_5 + \bar y_{10}, \bar y_5 \bar y_{10}] \bar e_{(\ga^2)} \otimes R_{\al_2}) \\
  &= \O[\bar y_5 + \bar y_{10}, \bar y_5 \bar y_{10}, \bar y_{11}] \bar e_\be.
\end{align*}
We also compute (cf. \cite[\S5]{McN}):
$$
(\bar y_5 + \bar y_{10}) \bar e(\bi_\be) = - \bar \psi_5 \bar \psi_6 \bar \psi_7 \bar \psi_8 \bar \psi_9 \bar \psi_{10}^2 \bar \psi_9 \bar \psi_8 \bar \psi_7 \bar \psi_6 \bar \psi_5 \bar e(\bi_\be),
$$
which is zero because it contains the weight $(12341234323)$, and this is not a weight of $L(\be)$. Since $\bar y_{11} \bar e(\bi_\be) = \bar y_{10}^2 \bar e(\bi_\be)$, we see that $\bar e_\be \bar R_\be \bar e_\be = \O[\bar y_{11}] \bar e_\be$, as required.


\subsubsection{Type $G_2$}
$\be = \al_1$: $\bi_\be = (1)$, $D_\be = \de_\be = e(\bi_\be)$, and $y_\be = y_1 e(\bi_\be)$.

$\be = \al_2$: $\bi_\be = (2)$, $D_\be = \de_\be = e(\bi_\be)$, and $y_\be = y_1 e(\bi_\be)$.

$\be = \al_1 + \al_2$: $\bi_\be = (12)$, $D_\be = \de_\be = e(\bi_\be)$, and $y_\be = y_1 e(\bi_\be)$.

$\be = 2\al_1 + \al_2$: $\bi_\be = (112)$, $D_\be = \psi_1 e(\bi_\be)$, $\de_\be = y_2 e(\bi_\be)$, and $y_\be = (y_1 + y_2) e(\bi_\be)$.


$\be = 3\al_1 + \al_2$: $\bi_\be = (1112)$, $D_\be = \psi_1 \psi_2 \psi_1 e(\bi_\be)$, $\de_\be = y_2 y_3^2 e(\bi_\be)$, and $y_\be = y_1 y_2 y_3 e(\bi_\be)$.

Let $\mu$ be the composition $R_{3\al_1} \boxtimes R_{\al_2} \into R_\be \onto \bar R_\be$. If $w \notin \Si_{3,1}$, then $\psi_w e(\bi_\be) \in I_{>(\be)}$, and so $\mu$ is surjective. Furthermore, $\bar e_\be = \mu(e_{(\al_1^3)} \otimes 1)$. Thus $\bar e_\be \bar R_\be \bar e_\be = \O[\bar y_1, \bar y_2, \bar y_3, \bar y_4]^{\Si_{3,1}} \bar e_\be$. Since $(\bar y_3^3 - \bar y_4) \bar e(\bi_\be) = \bar \psi_3^2 \bar e(\bi_\be) = 0$, we have that $\bar e_\be \bar R_\be \bar e_\be$ is generated by $\O[\bar y_1, \bar y_2, \bar y_3]^{\Si_3} \bar e_\be$. Observe using \cite[Theorem 4.12(i)]{KLM} that
$$
  (\bar y_1 + \bar y_2 + \bar y_3) \bar e_\be = \bar e_\be \bar \psi_1 \bar \psi_2 \bar \psi_3^2 \bar e_\be \in I_{>(\be)}
$$
and
$$
  ((\bar y_1 + \bar y_2 + \bar y_3)^2 - (\bar y_1 \bar y_2 + \bar y_1 \bar y_3 + \bar y_2 \bar y_3)) \bar e_\be = \bar e_\be \bar \psi_2 \bar \psi_3^2 \bar e_\be \in I_{>(\be)}.
$$
Therefore $\bar e_\be \bar R_\be \bar e_\be$ is generated by $\bar y_1 \bar y_2 \bar y_3$.

$\be = 3\al_1 + 2\al_2$: $\bi_\be = (11212)$, $D_\be = \psi_1 \psi_3 \psi_2 \psi_4 \psi_1 \psi_3 e(\bi_\be)$, $\de_\be = y_2 y_4^2 e(\bi_\be)$, and $y_\be = y_1 y_2 y_4 e(\bi_\be)$. We first prove 

\vspace{1 mm}
\noindent
{\em Claim:}
If $w \neq 1$ then $e(\bi_\be) \psi_w e_\be \in I_{>(\be)}$.

\vspace{1mm}
\noindent
This is clearly true unless $w$ is one of the twelve permutations that stabilizes the weight $\bi_\be$. Of these, six produce a negative degree. Since $D_\be$ spans the smallest degree component of $e(\bi_\be) R_\be e(\bi_\be)$ we the Claim holds for these six permutations. Two of the remaining six permutations end with the cycle $(12)$. Since $\psi_1 D_\be = 0$, this implies that the Claim holds for them too. Finally, reduced decompositions for the remaining non-identity permutations may be chosen so that $e(\bi_\be) \psi_w \in I_{>(\be)}$ by Lemmas~\ref{LBadWordsNew} and \ref{LHighestWt}. 

Now we combine the Claim with Theorem~\ref{TBasis} to see that 
$\bar e_\be \bar R_\be \bar e_\be$ is generated by $\bar e_\be \O[\bar y_1, \dots, \bar y_5] \bar e_\be$. Next, by weights and quadratic relations, $(y_3-y_2^3) e(\bi_\be), (y_5-y_4^3) e(\bi_\be) \in I_{>(\be)}$. Thus $\bar e_\be \bar R_\be \bar e_\be = \bar e_\be \O[\bar y_1, \bar y_2, \bar y_4] \bar e_\be$. This can then be seen to be equal to $\O[\bar y_1, \bar y_2, \bar y_4]^{\Si_3} \bar e_\be$, arguing as in the case of the root $3\al_1 + \al_2$ above. Again as in the case of the root $3\al_1 + \al_2$, one then shows using specific elements of $I_{>(\be)}$that $(\bar y_1 + \bar y_2 + \bar y_4) \bar e_\be = 0$ and $(\bar y_1 \bar y_2 + \bar y_1 \bar y_4 + \bar y_2 \bar y_4) \bar e_\be = 0$, so that $\bar e_\be \bar R_\be \bar e_\be = \O[\bar y_1 \bar y_2 \bar y_4] e_\be$.



\end{document}